\documentclass[12pts,a4paper]{amsart}
\usepackage{amsmath,amsfonts,amsthm,amssymb,mathtools}
\usepackage[left=1in,right=1in,top=1in,bottom=1in]{geometry}
\usepackage{dirtytalk}

\usepackage[dvipsnames]{xcolor}
\usepackage[pagebackref]{hyperref} 
\hypersetup{
	colorlinks=true,
	menucolor=red,
	linkcolor=blue,
	citecolor=BrickRed,
	filecolor=red,      
	urlcolor=MidnightBlue,
	pdftitle={Multivariate de Rham Comparison Theorem},
	pdfpagemode=UseNone,
	pdfstartview={XYZ null null 1.00}
}

\newtheorem{proposition}{Proposition}[section]
\newtheorem{lemma}[proposition]{Lemma}
\newtheorem{remark}[proposition]{Remark}
\newtheorem{corollary}[proposition]{Corollary}
\newtheorem{definition}[proposition]{Definition}
\newtheorem{theorem}[proposition]{Theorem}
\newtheorem{example}[proposition]{Example}

\DeclareMathOperator{\Hom}{Hom}
\DeclareMathOperator{\Frac}{Frac}
\DeclareMathOperator{\spm}{Spm}
\DeclareMathOperator{\spec}{Spec}
\DeclareMathOperator{\spa}{Spa}

\DeclareMathOperator{\Char}{Char}
\DeclareMathOperator{\Nil}{Nil}
\DeclareMathOperator{\pr}{pr}
\DeclareMathOperator{\Mod}{Mod}
\DeclareMathOperator{\alHom}{alHom}
\DeclareMathOperator{\Alg}{Alg}
\DeclareMathOperator{\Tor}{Tor}
\DeclareMathOperator{\Ext}{Ext}
\DeclareMathOperator{\id}{id}
\DeclareMathOperator{\Perf}{Perf}
\DeclareMathOperator{\B}{\textbf{B}}
\DeclareMathOperator{\G}{G}

\DeclareMathOperator{\Supp}{Supp}
\DeclareMathOperator{\fet}{f\acute{e}t}
\DeclareMathOperator{\cH}{H}

\DeclareMathOperator{\Bdr}{B_{dR,\Delta}}
\DeclareMathOperator{\BHT}{B_{HT,\Delta}}
\DeclareMathOperator{\Bcris}{B_{cris,\Delta}}
\DeclareMathOperator{\Bst}{B_{st,\Delta}}
\DeclareMathOperator{\C}{\textbf{C}}
\usepackage{tikz-cd}
\usetikzlibrary{decorations.pathreplacing,angles,quotes}

\begin{document}
	\title{Perfectoid Spaces in Multivariate $p$-adic Hodge Theory}

          \author{Aprameyo Pal}
	\address{(Aprameyo Pal) Harish-Chandra Research Institute, A CI of Homi Bhabha National Institute, Chhatnag Road, Jhunsi, Prayagraj - 211019, India}
	\email{aprameyopal@hri.res.in}

      		\author{Rohit Pokhrel} 
	\address{(Rohit Pokhrel) Harish-Chandra Research Institute, A CI of Homi Bhabha National Institute, Chhatnag Road, Jhunsi, Prayagraj - 211019, India}
	\email{rohitpokhrel@hri.res.in}
	
	\maketitle

\begin{abstract}
Perfectoid spaces have become a crucial tool in $p$-adic geometry, serving as a bridge between adic spaces in characteristic $0$ and those in characteristic $p$. In this article, we develop a systematic way to study the structure of perfectoid spaces within the setting of multivariate $p$-adic Hodge theory over a variant of the rings introduced in \cite{Bri}.
\end{abstract}
 
\tableofcontents
	
\section{Introduction}
\footnote{2020 \it Mathematics Subject Classification. 
\rm Primary: 14G45, 14G17, 14G20, 12J25, 11S37.}

\begin{subsection}{Overview}
   The theory of Multivariate $p$-adic Hodge Theory was initiated by a series of works by Gergely Zábrádi \cite{Zab,Zab2}. He developed multivariate $(\phi_\Delta,\Gamma_\Delta)$-modules, extending Fontaine's $(\phi,\Gamma)$-modules, to study $p$-adic representations of $\G_{K,\Delta}=\G_K\times\cdots\times \G_K$, which is the $t$-fold direct product of the absolute Galois group $\G_K$ of a $p$-adic field $K$ (here $t=|\Delta|$). While the direct product operation on the Galois group side seems harmless, the corresponding object on Fontaine's side isn't straightforward due to the involvement of tensor products. Tensor products of representations are typically chaotic and are frequently studied in the Langlands program. Possible applications of this theory are discussed in \cite{Pal-Zab,CarKedZab}.

  Furthermore, the objects involved in this theory appear intrinsically interesting and possess a multivariate flavor. One example hinting at multivariate objects is the work by Brinon, Chiarellotto, and Mazzari \cite{Bri}, where they developed classical objects (e.g., the period rings $\Bdr$ and $\BHT$) previously studied in Fontaine's classical work \cite{Fontaine}. Continuing this, Poyeton and Vanni \cite{poyeton} defined other period rings, $\Bcris$ and $\Bst$.

  In this article, we'll provide a cleaner way to study these objects. This approach makes our tools closely related to the classical theory, yet it also proves that the multivariate theory is fundamentally different because its objects are distinct. Moreover, this article will be a first step towards exploring the geometric side of multivariate $p$-adic Hodge Theory, which provides the core motivation for studying it in the classical setup.

  This geometric theory does not build upon Scholze's theory of diamonds as in \cite{CarKedZab}, but rather constructs the objects of interest within the multivariate framework itself.
\end{subsection}

\begin{subsection}{Notations}
    We fix the following notations throughout the article;
    \begin{itemize}
        \item $\mathbb{N}^*=\{1,2,\cdots\}$ the set of positive integers,
        \item $\mathbb{N}=\{0,1,\cdots\}$ the set of non-negative integers,
        \item $p$ is a fixed prime,
        \item $\Delta=\{1,\cdots,t\}$ a finite set with $t$ elements,
        \item $\mathcal{O}_K$ is the ring of integers of a non-Archimedean field $K$,
        \item $\hat{\bullet}$ is the completion of $\bullet$.
    \end{itemize}
\end{subsection}

\begin{subsection}{Main Results}
   Let $K_0$ be a $p$-adic field. We recall the following definition:
  \begin{definition}[\cite{Sch}, Definition 1.2]
  A perfectoid field is a complete topological field $K$ whose topology is induced by a non-discrete valuation of rank 1, such that the Frobenius $\Phi$ is surjective on $K^\circ/p$.
  \end{definition}
  Denote by $\C$ the completion of the algebraic closure of $K_0$. Let $K_1,\ldots,K_t$ be perfectoid fields over $K_0$ such that $K_0\subset K_\alpha\subset \C$. In \cite{Bri}, the authors consider $K_\alpha=\C$ for all $\alpha\in\Delta$ and $K_0=W(k_0)$, where $k_0$ is the residue field of $K_0$. Our basic ring of interest is
  \begin{align*}
  K_\Delta=K_1\hat{\otimes}_{K_0}\cdots\hat{\otimes}_{K_0}K_t.
  \end{align*}
  Here $\hat{\otimes}$ denotes the completed tensor product given in \cite{BGR}.
  It turns out that $K_\Delta$ is highly analogous to a perfectoid field in many properties. Specifically, it is a reduced ring of dimension zero whose localizations are all perfectoid fields, meaning every element is either a unit or a zero divisor. However, this ring is very different from a field in the sense that, excluding trivial cases, it has zero divisors and is not Noetherian. Because of its non-Noetherian nature, the adic space over $K_\Delta$ is very chaotic, and some techniques from classical theory \cite{Hubet} developed by Huber do not seem to hold in general. Thanks to the work of many people, this kind of problem is now taken care.

  The ring $K_\Delta$ is not strongly Noetherian; that is, $K_\Delta\left<T_1,\cdots,T_n\right>$ is not Noetherian for some $n$ (in fact, for $n=0$). However, we can conduct a thorough study of this ring by developing functional analysis over $K_\Delta$. In this regard, we define Banach spaces over $K_\Delta$ and prove an open mapping theorem for them. We believe most properties of Banach spaces can be translated to this category, but we are not pursuing that in this article to keep its length reasonable. Some important properties that a Tate algebra over a non-Archimedean field possesses are the division algorithm, the Weierstrass preparation theorem, and the fact that every ideal is closed. Yet, we do not know any analogue of these properties in this setup \cite{BGR,FormalRigid}. Consequently, we do not know whether the quotient of a Tate algebra possesses a Banach algebra structure in a multivariate context. Therefore, we only consider the case where the quotient has a Banach algebra structure in Theorem \ref{IntegralTateAcyclicity}; in other results, this assumption can be dropped.

  The inspiration for this article is an analogous figure to one given in Chapter 7 of \cite{Sch2} for a perfectoid field $K$ and its tilt $K^\flat$, extended to include the ring $K^\circ_\Delta$ and its tilt $K^{\circ\flat}_\Delta$. For a topological ring $A$, we denote $A^\circ$ to be the set of power-bounded elements.

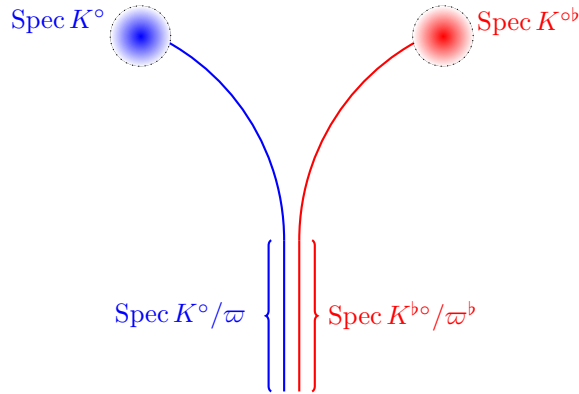
\begin{center}
    \begin{tikzpicture}
        \draw[blue,thick] (-0.1,0) -- (-0.1,2) ;
        \draw [blue,thick, decorate,decoration={brace,amplitude=2pt},xshift=-2pt] (-.2,0) -- (-0.2,2) node [blue,midway,xshift=-1.2cm] {$\spec K^\circ/\varpi$};
        \draw[blue,thick] (-0.1,2) arc (0:75:3cm) node[anchor=east]{$\spec K^\circ$};
        \draw[inner color=blue, outer color=white,fill,dotted] (-2.0,4.7) circle (.4cm);

        \draw[red,thick] (0.1,0) -- (0.1,2);
        \draw[red,thick] (0.1,2) arc (180:105:3cm) node[anchor=west]{$\spec K^{\circ\flat}$};
        \draw[inner color=red,outer color=white,fill,dotted] (2.0,4.7) circle (.4cm) ;
        \draw [red,thick, decorate,decoration={brace,amplitude=2pt,mirror},xshift=2pt] (0.2,0) -- (0.2,2) node [red,midway,xshift=1.2cm] {$\spec K^{\flat\circ}/\varpi^\flat$};
    \end{tikzpicture}
\end{center}

This diagram can be summarized by the following chain of equivalences:
\begin{align*}
    K_{\fet}\cong K^{\circ a}_{\fet}\cong (K^{\circ a}/\varpi)_{\fet}=(K^{\flat\circ a}/\varpi^\flat)_{\fet}\cong K^{\flat\circ a}_{\fet}\cong K_{\fet}.
\end{align*}
In words, this chain of equivalences states that the \'etale site corresponds to the \'etale site on the almost integral side, in the sense of Faltings' almost mathematics, which in turn corresponds to the almost \'etale site of its special point, which is equal to that of the tilted side. This can then be lifted back to obtain the \'etale site of the tilted almost integral level, which corresponds to the \'etale site of the tilted perfectoid field.

We also prove the analogous result for $K_\Delta$, and using the idea given in \cite{Jac}, we can prove that the ring $K_\Delta$ has one or infinitely many prime ideals indexed by a family $\mathfrak{S}$ defined below. We call $K_\Delta$ a perfectoid $\Delta$-field, although it is not a field in most cases. In fact, $K_\Delta$ is a perfectoid field if and only if $\mathfrak{S}$ is a singleton. Thus, one can say that $\mathfrak{S}$ measures the obstruction for $K_\Delta$ to be a field. The set of power-bounded elements $K^\circ_\Delta$ also has infinitely many minimal prime ideals indexed by $\mathfrak{S}$ and maximal ideals indexed by $\mathfrak{K}$ (defined below), and is of dimension 1. We also prove that the tilted side exhibits a similar phenomenon, and hence we obtain a diagram as above, fiberwise. Envoking the same strategy as given in \cite{Sch}, we get the analogous theorem:
\begin{theorem}
    The categories $(K_\Delta)_{\fet}$ and $(K^{\flat}_{\Delta})_{\fet}$ are equivalent. This runs through the following chain of equivalence
    \begin{align*}
        (K_\Delta)_{\fet}\cong (K^{\circ a}_\Delta)_{\fet}\cong (K^{\circ a}_\Delta/\varpi)_{\fet}= (K^{\flat\circ a}_\Delta/{\varpi^\flat})_{\fet}\cong (K^{\flat\circ a}_\Delta)_{\fet}\cong (K^\flat_\Delta)_{\fet}.
    \end{align*}
\end{theorem}

\end{subsection}

\begin{subsection}{Acknowledgment}
The second author acknowledges the support of the HRI Institute Fellowship for Research Scholars, and the first author acknowledges the SERB MATRICS grant (MTR/2022/000302). We would also like to thank Anand Uday Chitrao and Rakesh Pawar for their valuable discussions and suggestions, which greatly assisted us in this work.
\end{subsection}

\section{Objects in mixed characteristics setups}\label{section1}

Let $K_0$ be a $p-$adic field with ring of integers $\mathcal{O}_{K_0}$ and with a fixed algebraic closure $\overline{K}_0$. Assume all the valuations are normalized. Consider perfectoid fields $K_\alpha, \alpha\in\Delta$ such that $K_0\subseteq K_\alpha\subseteq \C$. Our basic ring of interest is $\mathcal{O}_{K,\Delta}=\mathcal{O}_{K_1}\hat{\otimes}_{\mathcal{O}_{K_0}}\cdots\hat{\otimes}_{\mathcal{O}_{K_0}}\mathcal{O}_{K_t}$. This is the completion of $\mathcal{O}:=\mathcal{O}_{K_1}\otimes_{\mathcal{O}_{K_0}}\cdots\otimes_{\mathcal{O}_{K_0}}\mathcal{O}_{K_t}$ with respect to the metric defined by the semi-norm

\begin{align*}
	|x|_\Delta:=\inf\bigl\{\max |x_1|\cdots|x_t| \mid x=\sum x_1\otimes\cdots\otimes x_t\bigr\}.
\end{align*}

Let $\varpi_\alpha$ be a pseudo-uniformizer of $K_\alpha$, for $\alpha\in\Delta$. Throughout the article we write $\varpi_\alpha$ for $1\otimes\cdots\otimes\varpi_\alpha\otimes\cdots\otimes 1$, $\varpi_\alpha$ in $\alpha-$th place and consider the ideal $I_\varpi$ generated by $\varpi_\alpha$ for $\alpha\in \Delta$.  We call this a pseudo-uniformizing ideal. 

\begin{proposition}
	$\mathcal{O}_{K,\Delta}$ is $I_\varpi-$adically complete for every pseudo-uniformizing ideal $I_\varpi$.
\end{proposition}

\begin{proof}
     We want to show
     \begin{align*}
     	 \mathcal{O}_{K,\Delta}&\to \varprojlim\limits_{n\in\mathbb{N}^*}\mathcal{O}_{K,\Delta}/I^n_{\varpi}\\
     	                                 x     &\mapsto (x\mod I^n_\varpi)_{n\in\mathbb{N}^*}
     \end{align*}
     is an isomorphism of topological rings. The kernel of the map is $\bigcap\limits_{n\in\mathbb{N}^*}I^n_{\varpi}$. Let $x\in\bigcap\limits_{n\in\mathbb{N}^*}I^n_\varpi$. For a particular $n$, we have $x\in I^n_\varpi$ then $x=\sum\limits_{i_1+\cdots+i_t\geq n}r_i \varpi_1^{i_1}\cdots\varpi^{i_t}_t$ with $r_i=r_{i1}\otimes\cdots\otimes r_{it}$ elementary tensors such that $|\varpi_\alpha|<|r_{i\alpha}|\leq 1$ in $K_\alpha$ and therefore $|x|_\Delta\leq |\varpi_1|^{r_1}\cdots |\varpi_t|^{r_t}\leq \max\limits_i|\varpi_i|^n$. Therefore, $|x|_\Delta=0$ giving $x=0$ proving $\bigcap\limits_{n\in\mathbb{N}^*}I^n_\varpi=0$. Hence, the above map is injective. To prove surjectivity, suppose $(\overline{x}_1,\overline{x}_2,\cdots)\in\varprojlim\limits_{n\in\mathbb{N}^*}\mathcal{O}_{K,\Delta}/I^n_\varpi$. Then for $m,n\in\mathbb{N}^*, m\geq n$ we have $x_m\equiv x_n\mod I^n_\varpi$ where $x_i$ is a lift of $\overline{x}_i\in\mathcal{O}_{K,\Delta}/I^i_\varpi$. Then $|x_m-x_n|_\Delta\leq \max\limits^t_{i=1}|\varpi_i|^n$ and hence $\{x_n\}$ is a Cauchy sequence which converges in $\mathcal{O}_{K,\Delta}$ whose image under the above map is an element that is equivalent to the given sequence. Therefore, $\mathcal{O}_{K,\Delta}$ is $I_\varpi-$adically complete.
\end{proof}

Let $J_\varpi$ be an ideal generated by family of pseudo-uniformizers $\{\varpi_1,\cdots,\varpi_t\}$ inside $\mathcal{O}$ and let $X_\varpi=\varprojlim\limits_{n\in\mathbb{N}^*}\mathcal{O}/J^n_\varpi$ be the $J_\varpi-$adic completion of $\mathcal{O}$. Define an ordering on the set of those families as $\varpi\leq \varpi'$ if and only if $X_{\varpi'}$ is $J_{\varpi}-$adically complete. Write 
\begin{align*}
	X:=\varinjlim\limits_\varpi X_\varpi
\end{align*}
Then it is clear that $X$ is complete with respect to any ideal generated by uniformizers $\{\varpi_1,\cdots,\varpi_t\}$ and is minimal with respect to that property in the sense that if $Y$ is an extension of $\mathcal{O}$ which is complete for every ideal generated by uniformizers $\{\varpi_1,\cdots,\varpi_t\}$, there is a unique continuous injection $X\to Y$ compatible with the direct system. 

\begin{proposition}
	The unique injective continuous map $X\to \mathcal{O}_{K,\Delta}$ is an isomorphism. 
\end{proposition}

\begin{proof}
Identifying with the image, we can assume that $X\subset\mathcal{O}_{K,\Delta}$. We only need to show that $X$ is complete with respect to the seminorm, which is sufficient due to the property mentioned above that $X$ satisfies. Let $\{x_n\}$ be a Cauchy sequence in $X$. For each $\alpha\in \Delta$, consider $\varpi_\alpha\in\mathcal{O}$. Then $|\varpi_\alpha|_\Delta>|\varpi^2_{\alpha}|_\Delta>\cdots$ decreasing to 0. Then, by considering a suitable subsequence, we can assume that
\begin{align*}
|x_{n+1} - x_n|_\Delta < \min\{|\varpi_1|^n, \ldots, |\varpi_t|^n\},
\end{align*}
i.e., \( x_{n+1} - x_n \in J^n_\varpi \), where \( J^n_\varpi \) is the ideal generated by \( \varpi_1, \ldots, \varpi_t \). Since \( X \) is \( J_\varpi \)-adically complete, there exists an element \( x \in X \) such that for each \( m \in \mathbb{N}^* \), there exists \( n \in \mathbb{N}^* \) with \( x - x_n \in J^m_\varpi \). Hence,
\begin{align*}
|x - x_n|_\Delta \to 0 \quad \text{as } n \to \infty.
\end{align*}
Therefore, $\lim\limits_{n\to\infty}x_n=x$ with respect to the metric $|\bullet|_\Delta$.
\end{proof}

In view of \cite{BGR} \S 2.1.7, the $L\hat{\otimes}_R$- is a left adjoint functor to the continuous $\Hom$ functor. The same argument as in the usual tensor product shows that
\begin{align*}
	\mathcal{O}_{K,\Delta}\left[\frac{1}{\varpi}\right]=K_1\hat{\otimes}_{K_0}\cdots\hat{\otimes}_{K_0} K_t=K_\Delta,
\end{align*}
which is the localization with respect to $S=\{1,\varpi,\varpi^2,\cdots\}$ for any uniformizer $\varpi\in K_0$. It is easy to check that $|\varpi x|_\Delta=|\varpi|_\Delta|x|_\Delta$ for all $x\in\mathcal{O}_{K,\Delta}$, so $|\bullet|_\Delta$ has a unique extension to $K_\Delta$. 

This shows that $K_\Delta$ is Huber with $\mathcal{O}_{K,\Delta}$ as its ring of definition and $I_\varpi$ as its ideal of definition. We can easily see that the norm $|\bullet|_\Delta$ on $K_\Delta$ is given by
\begin{align*}
	|x|_\Delta:=\inf\Big\{\max(|x_1|\cdots|x_t|)\mid x=\sum x_1\otimes\cdots\otimes x_t\Big\}.
\end{align*}

We will now classify all the prime ideals of $K_\Delta$ and $\mathcal{O}_{K,\Delta}$.

Let us consider a family 
\begin{align*}
	s=\big\{s_\alpha: K_\alpha\to \C \mid\text{ continuous embeddings over }K_0\big\}_{\alpha\in\Delta}
\end{align*}
and call it an embedding family and consider $K_s=\widehat{s_1(K_1)\cdots s_t(K_t)}$ be the completion of compositum of $s_1(K_1),\cdots,s_t(K_t)$ which is again a perfectoid field . Define an equivalence relation on the set of embedding families as $(s_\alpha)_{\alpha\in\Delta}=s\sim s'=(s'_\alpha)_{\alpha\in\Delta}$ if there exists a continuous isomorphism $u: K_s\to K_{s'}$ such that for all $\alpha\in\Delta$ the diagram
\[
\begin{tikzcd}
	K_s\arrow[rr,"u"] &  &  K_{s'}\\
	        & \arrow[ul,"s_\alpha"]  K_\alpha \arrow[ur,swap,"s'_\alpha"]   &
\end{tikzcd}
\]
commutes i.e. $s'_\alpha=u\circ s_\alpha$. 
 Denote by $\mathfrak{S}$ the set of all equivalence classes of embedding families. For such family $s\in\mathfrak{S}$, we have a surjective $K_0-$algebra homomorphism defined by
\begin{align*}
	\beta_s: K_\Delta &\to K_s\\
	       \sum x_1\otimes\cdots\otimes x_t&\mapsto \sum x^{s_1}_1\cdots x^{s_t}_t
\end{align*}
Then $\mathfrak{P}_s=\ker\beta_s$ is a maximal ideal of $K_\Delta$.

\begin{lemma}
	The ring homomorphism $\beta_s:K_\Delta\to K_s$ is continuous for every $s\in\mathfrak{S}$.
\end{lemma}

\begin{proof}
	Note that if $x=\sum\limits_{i=1}^\infty x_{i1}\otimes\cdots\otimes x_{it}$ is a representation of $x$, then
	\begin{align*}
		|\beta_s(x)|=|\sum\limits_{i=1}^\infty x^{s_1}_{i1}\cdots x^{s_t}_{it}|\leq \max\limits_{i=1}^\infty (|x_{i1}|\cdots |x_{it}|).
	\end{align*} 
	This is true for any representatives of $x$, which proves $|\beta_s(x)|\leq |x|_\Delta$. Therefore, $\beta_s$ is a continuous map for every $s\in\mathfrak{S}$.
\end{proof}

The following proposition classifies all the primes in $\spec K_\Delta$.

\begin{proposition}
	We have a one-to-one correspondence
	\begin{align*}
		\mathfrak{S}&\leftrightarrow \spec K_\Delta,\\
		                                                  s&\mapsto \mathfrak{P}_s.
	\end{align*}
	In particular, every prime ideal of $K_\Delta$ is maximal.
\end{proposition}

\begin{proof}
    We want to show that $s=(s_\alpha)_{\alpha\in\Delta} \mapsto \mathfrak{P}_s$ is well-defined. Let $(s_\alpha)_{\alpha\in\Delta}=s \sim s'=(s'_\alpha)_{\alpha\in\Delta}$, and recall that $x = \sum x_1 \otimes \cdots \otimes x_t \in \mathfrak{P}_s$ if and only if $\beta_s(x) = \sum x^{s_1}_1 \cdots x^{s_t}_t = 0$. Since for every $\alpha \in \Delta$, we have $s'_\alpha = u \circ s_\alpha$, we get
\begin{align*}
    \beta_{s'}(x) &= \sum x^{s'_1}_1 \cdots x^{s'_t}_t \\
                  &= \sum x^{u \circ s_1}_1 \cdots x^{u \circ s_t}_t \\
                  &= u\left(\sum x^{s_1}_1 \cdots x^{s_t}_t\right) \\
                  &= u(\beta_s(x)).
\end{align*}
Since $u$ is an isomorphism, we have $x \in \ker \beta_s \iff x \in \ker \beta_{s'}$, proving that $\mathfrak{P}_s = \mathfrak{P}_{s'}$. Thus, the above map is well-defined.

We will now prove surjectivity. Suppose $\mathfrak{P}$ is a prime ideal of $K_\Delta$. Then, for every $\alpha \in \Delta$, we get an injective map $s_\alpha$ as
\[
\begin{tikzcd}
    K_\alpha \arrow[r, "\iota_\alpha"] & K_\Delta \arrow[r] & K_\Delta / \mathfrak{P}
\end{tikzcd}
\]
where $\iota_\alpha$ is defined by $\iota_\alpha(x_\alpha) = 1 \otimes \cdots \otimes x_\alpha \otimes \cdots \otimes 1$ (with $x_\alpha$ in the $\alpha$-th place).

Let $x = \sum x_1 \otimes \cdots \otimes x_t \pmod{\mathfrak{P}}$; then $x = \sum s_1(x_1) \cdots s_t(x_t)$. Hence, $K_\Delta / \mathfrak{P}$ is a completion of the compositum of $s_\alpha(K_\alpha)$ for $\alpha \in \Delta$. This also shows that every prime ideal is maximal in $K_\Delta$. Therefore, $\mathfrak{P}$ gives a family of embeddings $s=(s_\alpha)_{\alpha\in\Delta}$. If $\mathfrak{P}_s=\mathfrak{P}_{s'}$, then we know that $K_\Delta/\mathfrak{P}_s$ and $K_\Delta/\mathfrak{P}_{s'}$ identify, respectively, with the completion of the compositum of the images $K_\alpha\to K_\Delta/\mathfrak{P}_s$ and $K_\alpha\to K_\Delta/\mathfrak{P}_{s'}$ for $\alpha\in\Delta$. Identifying the images, we obtain an isomorphism $u:K_\Delta/\mathfrak{P}_s\to K_\Delta/\mathfrak{P}_{s'}$, proving the equivalence.
\end{proof}

If $K_1$ and $K_2$ are perfectoid fields contained in $\C$. We claim that $\Frac(\mathcal{O}_{K_1}\mathcal{O}_{K_2})=K_1K_2$. First note that $\mathcal{O}_{K_1}\mathcal{O}_{K_2}\subseteq \mathcal{O}_{K_1K_2}$ and thus $\Frac(\mathcal{O}_{K_1}\mathcal{O}_{K_2})\subseteq K_1K_2$. If $\frac{a}{b}\in K_i$ is an arbitrary element with $a,b\in\mathcal{O}_{K_i}$ for $i=1,2$. Then $\frac{a}{b}\in\Frac(\mathcal{O}_{K_1}\mathcal{O}_{K_2})$ and thus $K_i\subset \Frac(\mathcal{O}_{K_1}\mathcal{O}_{K_2})$. This proves $K_1K_2=\Frac(\mathcal{O}_{K_1}\mathcal{O}_{K_2})$.\\

Suppose $s=(s_\alpha)_{\alpha\in\Delta}$ is a family of embeddings. It induces a local ring homomorphism $\beta_s: \mathcal{O}_{K_\alpha}\to \mathcal{O}_{K_s}$ which is just the restriction of $\beta_s$ defined above. Define an equivalence relation $s\sim s'$ families of embeddings as follows; The families of embeddings $s=(s_\alpha)_{\alpha\in\Delta}$ and $s'=(s'_\alpha)_{\alpha\in\Delta}$ induce surjective homomorphisms $\overline{s}=\{\overline{s}_\alpha: k_\alpha\to k_s\}$ and $\overline{s'}=\{\overline{s}_\alpha:k_\alpha\to k_s\}$. Say $s\sim_{\text{res}} s'$ if there exists an isomorphism $v: K_s\to K_{s'}$ inducing isomorphism $\overline{v}:k_s\to k_{s'}$ such that for all $\alpha\in\Delta$, the following diagram
\[
\begin{tikzcd}
	k_s \arrow[rr,"\overline{v}"]  &             &   k_{s'}\\
	                &  \arrow[ul,"\overline{s}_\alpha"]k_\alpha \arrow[ur,swap,"\overline{s'}_\alpha"] &
\end{tikzcd}
\]
commutes i.e, $\overline{s'}_\alpha=\overline{v}\circ \overline{s}_\alpha$. Let us denote its equivalence class by $\mathfrak{K}$, and call its elements a \emph{residual embedding family}. \\

Suppose $s=(s_\alpha)_{\alpha\in\Delta}$ consider the ring homomorphism
\begin{align*}
	\beta_s: \mathcal{O}_{K,\Delta}&\to \mathcal{O}_{K_s}\\
	                \sum x_1\otimes\cdots\otimes x_t &\mapsto \sum x^{s_1}_1\cdots x^{s_t}_t
\end{align*}
and $\mathfrak{p}_s=\ker\beta_s$ is a prime ideal of $\mathcal{O}_{K,\Delta}$. Taking the quotient by the maximal ideal $\mathfrak{m}_s$ of $K_s$, we get a ring homomorphism
\begin{align*}
	\overline{\beta}_s: \mathcal{O}_{K,\Delta}\to \mathcal{O}_{K_s}\to \mathcal{O}_{K_s}/\mathfrak{m}_s=k_s
\end{align*}
and $\mathfrak{m}_s=\ker\overline{\beta}_s$ is a maximal ideal of $\mathcal{O}_{K,\Delta}$. Also note that if $s\sim s'$ in $\mathfrak{S}$ then $s\sim_{\text{res}}s'$ in $\mathfrak{K}$ and therefore $\mathfrak{p}_s\subset \mathfrak{m}_s$.

\begin{proposition}\label{Prime ideal prop of Delta fields}
	The following gives one-to-one correspondences:
	\begin{align*}
		\mathfrak{S}&\leftrightarrow \spec\mathcal{O}_{K,\Delta}\setminus\spm\mathcal{O}_{K,\Delta}\\
		                  s &\mapsto \mathfrak{p}_s\\
	 \text{and  }	\mathfrak{K}&\leftrightarrow \spm\mathcal{O}_{K,\Delta}\\
		                  s &\mapsto \mathfrak{m}_s
	\end{align*}
\end{proposition}

\begin{proof}
	We will first show that the first map is well-defined. Let $s\sim s'$ be two equivalent embedding families in $\mathfrak{S}$, there exists an isomorphism $u: K_s\to K_{s'}$ such that $s'=u\circ s$ and same computation as in the previous theorem shows that $\mathfrak{p}_s=\mathfrak{p}_s$ and $\mathfrak{m}_s=\mathfrak{m}_s$. Suppose $\mathfrak{p}$ be a prime ideal $\mathcal{O}_{K,\Delta}$ then for each $\alpha\in\Delta$ we have 
	\[
	\begin{tikzcd}
	  \phi_\alpha:\mathcal{O}_{K_\alpha}\arrow[r,"\iota_\alpha"]& \mathcal{O}_{K,\Delta}\arrow[r]& \mathcal{O}_{K,\Delta}/\mathfrak{p}
	\end{tikzcd}
	\]
	where $\iota_\alpha$ is defined as in the previous proposition. Since $\mathfrak{p}$ be a prime ideal of $\mathcal{O}_{K,\Delta}$, then $\mathcal{O}_{K,\Delta}/\mathfrak{p}$ is an integral domain so that $\ker\phi_\alpha$ is a prime ideal, which has two possibilities $(0)$ and $\mathfrak{m}_{K_j}$. If for some distinct $i,j\in\Delta$, we have $\ker\phi_i=(0)$ and $\ker\phi_j=\mathfrak{m}_{K_j}$. Then $k_j\hookrightarrow\Frac(\mathcal{O}_{K,\Delta}/\mathfrak{p})\hookleftarrow K_i$ which is not possible in mixed characteristics.\\
	\textbf{Case-I :} $\ker\phi_\alpha=(0), \forall\alpha\in\Delta:$ In this case $\phi_\alpha$ induces a map
	\begin{align*}
		\phi_\alpha:K_\alpha\to \Frac(\mathcal{O}_{K,\Delta}/\mathfrak{p})
	\end{align*} 
	for all $\alpha\in\Delta$. The image generate the compositum of some conjugates of $K_\alpha$ as in the previous proposition, giving an embedding family in this case.\\
	\textbf{Case-II :} $\ker\phi_\alpha=\mathfrak{m}_{K_\alpha},\forall \alpha\in\Delta$ : Then $\phi_\alpha$ induces an injective maps
	\begin{align*}
		k_\alpha\to \Frac(\mathcal{O}_{K,\Delta}/\mathfrak{p})
	\end{align*}
	for all $\alpha\in\Delta$. The image generates the compositum of some conjugates of $k_\alpha$, giving a residual embedding family. But this residual embedding family comes from an embedding family. Injectivity is similar to the previous proposition. This proves the theorem.
\end{proof}

\begin{remark}
    The same proof also shows that $\spec (K_1\otimes_{K_0}\cdots\otimes_{K_0} K_t)$ is in one-one correspondences with $\mathfrak{S}$. Similarly, for the tensor product of corresponding integral subrings.
\end{remark}

\begin{remark}
    In general, $\mathfrak{S}$ and $\mathfrak{K}$ may be different. For example, consider $\mathbb{Q}_p(\mu_{p^\infty})\hat{\otimes}_{\mathbb{Q}_p}\mathbb{Q}_p(\mu_{p^\infty})$. Since $\mathbb{Q(\mu_{p^\infty})}$ is totally ramified over $\mathbb{Q}$, its residue field is $\mathbb{F}_p$. Therefore, $\mathfrak{K}$ in this case is a singleton, but we will construct two embedding families which are not equivalent. Define $s=(s_1,s_2)$ by defining $s_1(\zeta_{p^n})=\zeta^2_{p^n}$ and $s_2=\id_{\mathbb{Q}_p(\mu^\infty})$ and $s'=(s'_1,s'_2)=(\id_{\mathbb{Q}_p(\mu^\infty}),\id_{\mathbb{Q}_p(\mu^\infty}))$. Then it can be easily seen that $s$ and $s'$ are not equivalent, but $\overline{s}=(\id_{\mathbb{F}_p},\id_{\mathbb{F}_p})$ and $\overline{s'}=(\id_{\mathbb{F}_p},\id_{\mathbb{F}_p})$.  
\end{remark}
 
We will now introduce a new semi-norm in $\mathcal{O}_{K,\Delta}$ which behaves a little better with the multiplication as compared to $|\bullet|_\Delta$. Define
\begin{align*}
	|x|=\sup\limits_{s\in\mathfrak{S}}|\beta_s(x)|
\end{align*}
where the norm $|\beta_s(x)|$ is taken in $K_s$. We always assume that all norms are normalized.
\begin{align*}
	|x+y|&=\sup\limits_{s\in\mathfrak{S}}\{|\beta_s(x+y)|\}\\
	        &=\sup\limits_{s\in\mathfrak{S}}\{|\beta_s(x)+\beta_s(y)|\}\\
	        &\leq\sup\limits_{s\in\mathfrak{S}}\Big\{\max\{|\beta_s(x)|,|\beta_s(y)|\}\Big\}\\
	        &\leq \max\Big\{\sup\limits_{s\in\mathcal{S}}\{|\beta_s(x)|\},\sup\limits_{s\in\mathcal{S}}\{|\beta_s(y)|\}\Big\}\\
	        &=\max\{|x|,|y|\}
\end{align*}
Thus $|x+y|\leq\max\{|x|,|y|\}$ for every $x,y\in\mathcal{O}_{K,\Delta}$. Also,
\begin{align*}
	|xy|&=\sup\limits_{s\in\mathfrak{S}}\{|\beta_s(xy)|\}\\
	     &=\sup\limits_{s\in\mathfrak{S}}\{|\beta_s(x)||\beta_s(y)|\}\\
	     &\leq \sup\limits_{s\in\mathfrak{S}}\{|\beta_s(x)|\}\sup\limits_{s\in\mathfrak{S}}\{|\beta_s(y)|\}\\
	     &=|x||y|.
\end{align*}
But if for $x \in \mathcal{O}_{K,\Delta}$ we have $|x| = 0$, then $x$ is in the nilradical of $\mathcal{O}_{K,\Delta}$, i.e., $x^n = 0$. Thus, in order to prove that this is a seminorm, we have to show that $\mathcal{O}_{K,\Delta}$ is reduced.

\begin{proposition}
	Suppose $K_1,\cdots,K_t$ are algebraic extensions of an arbitrary non-Archimedean field $K_0$. Then $K_1\hat{\otimes}_{K_0}\cdots\hat{\otimes}_{K_0}K_n$ is a reduced ring iff they are linearly disjoint over $K_0$ or for every $1\leq i<j\leq t$, $K_i\cap K_j$ is a seperable extension of $K_0$.
\end{proposition}

\begin{proof}
	If $K_1,\dots, K_t$ are linearly disjoint over $K_0$, then we have $K_1\hat{\otimes}_{K_0}\cdots\hat{\otimes}_{K_0}K_t \cong \widehat{K_1\cdots K_t}$, which is therefore reduced. Suppose $K_i\cap K_j$ is not separable over $K_0$; then there exists an element $x\in K_i\cap K_j\setminus K_0$ such that $x^{p^k}\in K_0$. Then,
	\begin{align*}
		y=1\otimes\cdots\otimes x\otimes\cdots\otimes 1-1\otimes\cdots\otimes x\otimes\cdots\otimes 1
	\end{align*}
	where the first term has $x$ in the $i$-th place and the second term has $x$ in the $j$-th place. If $\Char K_0=p$, then $y^{p^k}=0$, and hence $K_\Delta$ is not reduced. Therefore, if $K_\Delta$ is reduced, then either $K_1,\dots,K_n$ are linearly disjoint over $K_0$ or $K_i\cap K_j$ is separable over $K_0$ for every $1\leq i<j\leq t$.\\
    
    Conversely, suppose for every $1\leq i<j\leq t$, $K_i\cap K_j$ is separable over $K_0$.Then by Bourbaki Algebra-II\cite{Bou} (Chapter 5, \S 15), we have that $K_1\otimes_{K_0}\cdots\otimes_{K_0}K_n$ is reduced. We will now show that $K_\Delta$ is reduced. Suppose $x=\sum\limits_{i=1}^\infty x_{i1}\otimes\cdots\otimes x_{it}$ is a nilpotent element of $K_\Delta$. Since for each $r\in\mathbb{R}_{>0}$ the number of terms $y_i:=x_{i1}\otimes\cdots \otimes x_{in}$ such that $|y_i|\geq r$ is finite, we can rearrange and write $x=\sum\limits_{i=1}^{i_1-1}y_i+\sum\limits_{i=i_1}^{i_2-1}y_i+\cdots$ where for $i_k\leq i<i_{k+1}$ we have $|y_i|=r_k$ with $r_1>r_2>\cdots$. Then $\beta_s(x)=0$ for every $s\in\mathfrak{S}$, and thus by the continuity of $\beta_s$, we have
	\begin{align*}
		\beta_s(x)=\left(\sum\limits_{i=1}^{i_1-1}\beta_s(y_i)\right)+\left(\sum\limits_{i=i_1}^{i_2-1}\beta_s(y_i)\right)+\cdots=0.
	\end{align*}
	Then $|\beta_s(y_i)|=|s_1(x_{i1})\cdots s_t(x_{it})|=|s_1(x_{i1})|\cdots |s_t(x_{it})|=|x_{i1}|\cdots|x_{it}|=r_k$ whenever $i_k\leq i<i_{k+1}$ and for every $s\in\mathfrak{S}$. But since $r_1>r_2>\cdots$ and since $|x+y|=0$ implies $|x|=|y|$ for any non-Archimedean field, we must have $\sum\limits_{i=i_k}^{i_{k+1}-1}\beta_s(y_i)=0$ for every $k$, i.e., the terms with same absolute value should add up to zero. Thus, each of these terms is nilpotent in $K\otimes_{K_0}\cdots\otimes_{K_0}K_n$, (as $x$ lies in all the prime ideal of $K_1\otimes_{K_0}\cdots\otimes_{K_0} K_t$ whose prime ideals are again indexed by $\mathfrak{S}$) which are all therefore zero since the ring $K_1\otimes_{K_0}\cdots\otimes_{K_0} K_t$ is reduced.
\end{proof}

Finally, we can conclude that the map $|\bullet|:K_\Delta\to\mathbb{R}_{\geq 0}$ defines a seminorm on $\mathcal{O}_{K,\Delta}$ made from perfectoid fields $K_\alpha,\alpha\in\Delta$ over $K_0$.

\begin{proposition}
	For every $x\in\mathcal{O}_{K,\Delta}$, we have $|x^n|=|x|^n$.
\end{proposition}

\begin{proof}
   Suppose $|x|=r$, then for each $\epsilon>0$, there exists $s\in\mathfrak{S}$ such that $r-\epsilon< |\beta_s(x)|\leq r$. Then $(r-\epsilon)^n<|\beta_s(x)|^n=|(\beta_s(x))^n|=|\beta_s(x^n)|\leq r^n$. But $(r-\epsilon)^n=r^n+\epsilon *$, and as $\epsilon$ is arbitrary, we get $|x^n|=\sup\limits_{s\in\mathfrak{S}}|\beta_s(x^n)|=r^n=|x|^n$, which completes the proof.
\end{proof}
We note that the semi-norm $|\bullet|$ can be extended to $K_\Delta = \mathcal{O}_{K,\Delta}\left[\frac{1}{\varpi_0}\right]$.

\begin{proposition}
    The metrics $|\bullet|_\Delta$ and $|\bullet|$ on $K_\Delta$ are equivalent.
\end{proposition}

\begin{proof}
 Let $x = \sum_{i=1}^\infty x_{i1} \otimes \cdots \otimes x_{it}$ be a representation of $x$ in $K_\Delta$; then for every $s \in \mathfrak{S}$ we have:
\begin{align*}
 |\beta_s(x)| = \left|\sum_{i=1}^\infty s_1(x_{i1}) \cdots s_t(x_{it})\right| \leq \max_{i=1}^\infty \{|s_1(x_{i1})| \cdots |s_t(x_{it})|\} = \max_{i=1}^\infty \{|x_{i1}| \cdots |x_{it}|\}.
\end{align*}
Since the right-hand side is independent of $s$, we obtain:
\begin{align*}
  |x| \leq |x|_\Delta.
\end{align*}
Thus, the identity map
\begin{align*}
\id : (K_\Delta, |\bullet|_\Delta) \to (K_\Delta, |\bullet|)
\end{align*}
is continuous and bijective. We will now prove that $(K_\Delta, |\bullet|)$ is complete. Let $(x_n)_{n\in\mathbb{N}}$ be a Cauchy sequence in $(K_\Delta, |\bullet|)$. For every $\epsilon > 0$, there exists $N \in \mathbb{N}$ such that $|x_n - x_m| < \epsilon$ for all $n, m > N$, and hence for every $s \in \mathfrak{S}$, we have $|\beta_s(x_n - x_m)| < \epsilon$. 

Suppose $x_n - x_m = \sum_{i=1}^\infty y_{i1} \otimes \cdots \otimes y_{it}$ is a representation of $x_n - x_m$. If $I \subset \mathbb{N}$ is a finite set of indices $i \in \mathbb{N}$ such that $|y_{i1}| \cdots |y_{it}| \geq \epsilon$, then for every $s \in \mathfrak{S}$ we have $|\beta_s(\sum_{i \in I} y_{i1} \otimes \cdots \otimes y_{it})| = 0$. Consequently, $\sum_{i \in I} y_{i1} \otimes \cdots \otimes y_{it} = 0$ as this term belongs to the nilradical. This proves that $(x_n)$ is Cauchy in $(K_\Delta, |\bullet|_\Delta)$, and since $(K_\Delta, |\bullet|_\Delta)$ is complete, $(x_n)$ converges in $(K_\Delta, |\bullet|_\Delta)$ and thus also converges in $(K_\Delta, |\bullet|)$. Since both $(K_\Delta, |\bullet|)$ and $(K_\Delta, |\bullet|_\Delta)$ are Banach spaces over $K_0$, the Open Mapping Theorem implies this map is a homeomorphism.
\end{proof}

\begin{remark}
	From now onwards, we will use the seminorm $|\bullet|$ as it behaves well with multiplication and is also they are closely related to the adic spectrum of $\spa(K_\Delta,K^+_\Delta)$, which we will use later.
\end{remark}

Note that $\mathcal{O}_{K,\Delta}$ is the ring $K^\circ_\Delta$ of power bounded elements in $K_\Delta$, c.f. \cite{Wed}. For simplicity, from now onwards we will denote $\mathcal{O}_{K,\Delta}$ by $K^\circ_\Delta$.

\begin{definition}
	For $I_\varpi$ the pseudo-uniformizing ideal of $K^\circ_\Delta$ defined by the pseudo-uniformizers $\varpi_\alpha$ of $K_\alpha$ satisfying $|p|\leq |\varpi_\alpha|<1$ $(1\leq\alpha\leq t)$, we define the tilt of $K^\circ_\Delta$ as 
	\begin{align*}
		K^{\circ\flat}_\Delta=\varprojlim\limits_{x\mapsto x^p}K^\circ_\Delta/I_\varpi:=\left\{(x_n)_{n\in\mathbb{N}}\in \prod\limits_{n\in\mathbb{N}}(K^\circ_\Delta/I_\varpi)\mid \text{ for all $n\in\mathbb{N}$ we have } x^p_{n+1}=x_n\right\}.
	\end{align*}
\end{definition}
This is a perfect $\mathbb{F}_p$-algebra with defining ring homomorphism $t\mapsto (tx_0,tx_1,\cdots)$ from $\mathbb{F}_p$ to $K^{\circ\flat}_\Delta$. 

\begin{lemma}
    If $x\equiv y\mod I^n_\varpi$, then $x^p\equiv y^p\mod I^{n+1}_\varpi$.
\end{lemma}

\begin{proof}
    We have $x=y+a$ for some $a\in I^n_\varpi$, then $x^p=y^p+\sum\limits_{i=0}^{p-1} {p\choose i}y^ia^{p-1}\equiv y^p\mod I^{n+1}_\varpi$ because $p\in I_\varpi$.
\end{proof}

\begin{lemma}
	There is a multiplicative bijection
	\begin{align*}
		K^{\circ\flat}_\Delta\to \varprojlim\limits_{x\mapsto x^p}K^\circ_\Delta
	\end{align*}
\end{lemma}

\begin{proof}
	Define
	\begin{align*}
		K^{\circ\flat}_\Delta&\to\varprojlim\limits_{x\mapsto x^p}K^\circ_\Delta\\
		          (x_n)_{n\in\mathbb{N}}&\mapsto (x^{(n)})_{n\in\mathbb{N}}
	\end{align*}
	with $x^{(n)}:=\varprojlim\limits_{m\to\infty }\tilde{x}^{p^m}_{n+m}$ where $(\tilde{x}_n)_{n\in\mathbb{N}}$ lifts $(x_n)_{n\in\mathbb{N}}$ inside $K^\circ_\Delta$. Let $(\tilde{x}'_n)_{n\in\mathbb{N}}$ be another lift, then $\tilde{x}'_{n+m}\equiv \tilde{x}_{n+m}\mod I_\varpi$, then $\tilde{x}'^{p^m}_{n+m}\equiv \tilde{x}^{p^m}_{n+m}\mod I^{m}_\varpi$ for all $m$. Thus $x^{(n)}$ is independent of the lift and satisfies the compatibility conditions.
\end{proof}

\begin{remark}
   \begin{enumerate}
       \item[(i)] This lemma shows that the definition of tilt is independent of the choice of pseudo-uniformizing ideal $I_\varpi$.

       \item[(ii)] For each $\alpha\in\Delta$, The natural ring homomorphism $\iota_\alpha: K^\circ_\alpha\to K^\circ_\Delta$, induces homomorphism 
       \begin{align*}
           \iota^\flat_\alpha: K^{\circ\flat}_\alpha&\to K^{\circ\flat}_\Delta,\\
                        (x^{(n)})_{n\in\mathbb{N}}  &\mapsto \iota_\alpha(x^{(n)})_{n\in\mathbb{N}}.
       \end{align*}
       Using Lemma 3.4 of \cite{Sch}, there exists elements $\varpi^\flat_\alpha$, $\alpha\in\Delta$ in $K^{\circ\flat}_\Delta$ such that $(\varpi^\flat_\alpha)^\#$ and $\varpi_\alpha$ differs by a unit in $K^\circ_\Delta$. Thus replacing $\varpi_\alpha$ by $(\varpi^\flat)^\#$, we have an element which has all $p^n$-th roots. For $p$, we write the corresponding elements as $\pi_\alpha$, $\alpha\in\Delta$.
   \end{enumerate}
    
\end{remark}

Transporting the structure, we can define the ring structure on $\varprojlim\limits_{x\mapsto x^p}K^\circ_\Delta$ via the following formulas:
\begin{align*}
	(x+y)^{(n)}&:=\lim\limits_{m\to\infty}(x^{(n+m)}+y^{(n+m)})^{p^m}\\
	  (xy)^{(n)}&:=x^{(n)}y^{(n)} .
\end{align*}
Thus we identify $K^{\circ\flat}_\Delta$ with the ring $\varprojlim\limits_{x\mapsto x^p}K^\circ_\Delta$. We have a natural multiplicative map $\#:K^{\circ\flat}_\Delta\to K^\circ_\Delta$ defined as $x=(x^{(n)})\mapsto x^\#=x^{(0)}$.

Define the seminorm $|\bullet|_\flat$ on $K^{\circ\flat}_\Delta$ by the formula
\begin{align*}
	|x|_\flat:=|x^\#|.
\end{align*}
We want to show the following
\begin{align*}
	|x+y|_\flat\leq\max\{|x|_\flat,|y|_\flat\}
\end{align*}
for all $x,y\in K^{\circ\flat}_\Delta$. To see this
\begin{align*}
	|x+y|_\flat&=|\lim\limits_{m\to\infty}(x^{(m)}+y^{(m)})^{p^m}|\\
	                &=\lim\limits_{m\to\infty}|x^{(m)}+y^{(m)}|^{p^m}\\
	                &\leq \lim\limits_{m\to \infty}\max\{|x^{(m)}|^{p^m},|y^{(m)}|^{p^m}\}\\
	                &=\lim\limits_{m\to\infty}\max\{|(x^{(m)})^{p^m}|,|(y^{(m)})^{p^m}|\}\\
	                &=\max\{|x^\#|,|y^\#|\}\\
	                &=\max\{|x|_\flat,|y|_\flat\}.
\end{align*}
and
\begin{align*}
    |xy|_\flat=|x^\#y^\#|\leq |x^\#||y^\#|=|x|_\flat|y|_\flat
\end{align*}
Therefore, $|\bullet|_\flat$ is a metric on $K^{\circ\flat}_\Delta$. From the power multiplicativity of $|\bullet|$, it is easy to check that $|\bullet|_\flat$ is power multiplicative. Indeed, for $x\in K^{\circ\flat}_\Delta$ we have $|x^n|_\flat=|(x^n)^\#|=|(x^\#)^n|=|x^\#|^n=|x|^n_\flat$.

\begin{proposition}
	The ring $K^{\circ\flat}_\Delta$ is reduced.
\end{proposition}

\begin{proof}
     If $x=(x^{(m)})_{m\in\mathbb{N}}$ is nilpotent then $x^n=0$ for some $n\in\mathbb{N}^*$. Then $(x^\#)^n=0$ and hence $x^{(0)}=x^\#=0$. Then $(x^{(1)})^p=x^{(0)}=0$ in $K^\circ_\Delta$ which is reduced, and thus $x^{(1)}=0$. Therefore, by induction $x^{(m)}=0$ for every $m\in\mathbb{N}$ and hence $x=0$, proving $K^{\circ\flat}_\Delta$ is reduced.
\end{proof}

Let $I_\varpi$ be the pseudo-uniformizing ideal defined by $\varpi_\alpha, \alpha\in\Delta$, then we consider $\varpi^\flat_\alpha\in K^{\circ\flat}_{\Delta}$ such that $(\varpi^\flat_\alpha)^\#=\varpi_\alpha$. Consider $I^\flat_\varpi$ be the ideal generated by $\varpi^\flat_\alpha$ for all $\alpha\in\Delta$ and call this also the pseudo-uniformizing ideal of $K^{\circ\flat}_\Delta$.

\begin{proposition}
	$K^{\circ\flat}_\Delta$ is complete with respect to the semi-norm $|\bullet|_\flat$.
\end{proposition}

\begin{proof}
	We will first prove that the completeness of $K^{\circ\flat}_\Delta$ with respect to $|\bullet|_\flat$ is equivalent to $I^\flat_\varpi-$adic completeness for any pseudo-uniformizing ideal $I^\flat_\varpi$. Suppose $K^{\circ\flat}_\Delta$ is complete with respect to $|\bullet|_\flat$ and consider the natural map
    \begin{align*}
      K^{\circ\flat}_\Delta&\to \varprojlim\limits_m K^{\circ\flat}_\Delta/(I^\flat_\varpi)^m.
    \end{align*}
    We seek to show this is an isomorphism of rings. The kernel of this map is $\cap(I^\flat_\varpi)^m$ and suppose $x\in \cap (I^\flat_\varpi)^m$, then by construction of $I^\flat_\varpi$, we have $x^\#\in\cap I^m_\varpi=\{0\}$ and hence $x^\#=x^{(0)}=0$. Assume that $x^{(n)}=0$, then $(x^{(n+1)})^p=x^{(n)}=0$ in $K^\circ_\Delta$ which is reduced. Therefore, $x^{(n)}=0$ by induction. Therefore, the map is injective.\\
    
    For surjectivity, consider $(x_n)_{n\in\mathbb{N}}\in \varprojlim\limits_n K^{\circ\flat}_\Delta/(I^\flat_\varpi)^n$, then for any $m\geq n$, we have $x_m\equiv x_n\mod(I^\flat_\varpi)^n$. Then we have $|x_m-x_n|\leq \max\limits_{\alpha\in\Delta}|\varpi^\flat_\alpha|^n_\flat=\max\limits_{\alpha\in\Delta}|\varpi_\alpha|^n$ for all $m\geq n$. Therefore, $(x_n)_{n\in\mathbb{N}}$ is a Cauchy sequence in $K^{\circ\flat}_\Delta$ with respect to $|\bullet|_\flat$ which converges to an element $x$ in $K^{\circ\flat}_\Delta$. By construction we have $x-x_n\in (I^\flat_\varpi)^n$ and hence $x$ maps to $\{x_n\}$ along the above map.\\
    
    Conversely, suppose $K^{\circ\flat}_\Delta$ is $I^\flat_\varpi-$adically complete for a pseudo-uniformizing ideal $I^\flat_\varpi$. Let $(x_n)_{n\in\mathbb{N}}$ be a Cauchy sequence in $K^{\circ\flat}_\Delta$. Note that for every $\alpha\in\Delta$ we have $|\varpi^\flat_\alpha|_\flat>|\varpi^\flat_\alpha|_\flat^2>\cdots$ decreases to $0$, so without loss of generality (by considering a suitable subsequence) we can assume $x_{m+n}\equiv x_m\mod (I^\flat_\varpi)^m$ and hence their images are an element of the inverse limit. By completeness for $I^\flat_\varpi$-adic topology, there exists an element $x\in K^{\circ\flat}_\Delta$ such that $x\equiv x_m\mod(I^\flat_\varpi)^m$. Therefore $\lim\limits_{m\to\infty}x_n=x$.\\
	
	Using the above, we can choose any $\varpi_\alpha\in K^\circ_\Delta$, and hence we choose $\varpi_\alpha=p$ for all $\alpha\in\Delta$ . The ideal $I_\varpi$ is then the ideal $(p)$ and the ideal $I^\flat_\varpi$ is generated by the elements $\pi_\alpha,\alpha\in\Delta$ which are made from a compatible system of roots of $p$ in $K^{\circ\flat}_\Delta$. This line is just an abuse of notation, in the sense that $p$ may not contain all $p^n$-th power roots, but multiplying by suitable units, we obtain such elements which generates the same ideal $pK^\circ_\Delta$. \\
	
	For $m\geq 0$, the maps of sets 
	\begin{align*}
		K^{\circ\flat}_\Delta&\to K^\circ_\Delta\\
		x=\{x^{(k)}\}_{k\in\mathbb{N}}&\mapsto x^{(m)}=(x^{1/p^m})^{(0)}
	\end{align*}
	induces a ring homomorphism $K^{\circ\flat}_\Delta\to K^\circ_\Delta/pK^\circ_\Delta$ which annihilates $\pi^{p^m}_\alpha$ for every $\alpha\in\Delta$ and therefore factors through a map $u_m:K^{\circ\flat}_\Delta/(I^\flat_p)^{p^m}\to K^\circ_\Delta/pK^\circ_\Delta$. Thus, we get a commutative diagram 
	\[
	\begin{tikzcd}
		\cdots \arrow[r]& K^{\circ\flat}_\Delta/(I^\flat_p)^{p^2}\arrow[r]\arrow[d,"u_2"] & K^{\circ\flat}_\Delta/(I^\flat_p)^p\arrow[r]\arrow[d,"u_1"] & K^{\circ\flat}_\Delta/I^\flat_p\arrow[d,"u_0"]\\
		\cdots\arrow[r]& K^\circ_\Delta/pK^\circ_\Delta\arrow[r,"\varphi"] & K^\circ_\Delta/pK^\circ_\Delta\arrow[r,"\varphi"] & K^\circ_\Delta/pK^\circ_\Delta
	\end{tikzcd}.
	\]
	The inverse limit of the bottom row is $K^\flat_\Delta$, and thus it suffices to show each $u_m$ is an isomorphism. This reduces to showing $u_0$ is an isomorphism, which is immediate.
\end{proof}

\begin{lemma}\label{lemma2.3}
	Consider an inverse system of rings $(f_i:A_i\to A_{i-1})_{i\geq 1}$ with each $f_i$ surjective and for each $i\geq 0$, $I_i\subset A_i$ is an ideal such that the following diagram
	\[
	\begin{tikzcd}
	\cdots \arrow[r,"f_3"] & A_2\arrow[r,"f_2"]                          & A_1\arrow[r,"f_1"]                           & A_0\\
	\cdots\arrow[r,"f_3"]  & I_2\arrow[r,"f_2"] \arrow[u,hook] & I_1\arrow[r,"f_1"]   \arrow[u,hook] & I_0\arrow[u,hook]
	\end{tikzcd}
	\]
	commutes. Then:
	\begin{enumerate}
		\item[(i).] $I:=\varprojlim\limits I_i$ is an ideal of $A:=\varprojlim A_i$.
		\item[(ii).] If $J$ is a closed ideal of $A$ with respect to inverse limit topology, where all $A_i$ are given the discrete topology, then $J=\varprojlim I_i$, for some inverse system of ideals such that the above diagram commutes.
		\item[(iii).] Prime ideals that are closed with respect to inverse limit topology correspond to an inverse system of prime ideals.
	\end{enumerate}
\end{lemma}

\begin{proof}
	(i). Let $(x_i)\in I$ and $(\lambda_i)\in A$ be arbitrary elements. Then we want to show $(\lambda_ix_i)\in I$. Since $\lambda_ix_i\in I_i$, it is enough to show $f_i(\lambda_ix_i)=f_i(\lambda_i)f_i(x_i)\in I_{i-1}$ which is immediate. Thus $I$ is an ideal of $A$.\\
    
	(ii). If $J$ is an ideal of $A$, then the projection maps $\pr_i:A\to A_i$ takes $J$ to an ideal $J_i$ for each $i\geq 0$. Also we have $f_{i+1}(J_{i+1})\subset J_i$. Indeed, $(f_{i+1}\circ \pr_{i+1})(y)=\pr_i(y)$ for each $y\in J$ and every $i\geq 0$. Therefore, $J\subseteq \varprojlim J_i$.\\

    Conversely, let $x=(x_i)_{i\in\mathbb{N}}$ be an arbitrary element of $\varprojlim J_i$. An arbitrary fundamental system of neighborhoods of $x$ is given by
    \begin{align*}
        U=\{y\in A\mid \pi_i(y)=x_i\text{ for } i\in F\}
    \end{align*}
    where $F$ is a finite subset of $\mathbb{N}$. Consider $k\in\mathbb{N}$ such that $k\geq i$ for all $i\in F$. Because $x\in\varprojlim J_i$, we have $x_k\in J_k=\pr_k(J)$ and thus there exists $y^{(k)}\in J$ such that $\pr_k(y^{(k)})=x_k$. This proves that, for any $i\leq k$, we have
    \begin{align*}
        \pr_i(y^{(k)})=f_{k-i+1}\circ\cdots\circ f_{k}(\pi_{k}(y^{(k)}))=x_i
    \end{align*}
    and hence $y^{(k)}\in U$. This shows that $x$ is in the closure of $J$. Since $J$ is closed $\varprojlim J_i\subset J$.\\
    
	(iii). Let $\mathfrak{p}=\varprojlim\mathfrak{p}_i$ be a closed prime ideal of $A$. Since $\pr_i$ is surjective, its image $\mathfrak{p}_i$ is also a prime ideal. Conversely, suppose each $\mathfrak{p}_i$ is  prime ideal and take $x=(x_i),y=(y_i)\in A$ be such that $x,y\in\mathfrak{p}$. Then for each $i\geq 1$ we have $x_iy_i\in\mathfrak{p}_i$. Since $\mathfrak{p}_i$ is a prime ideal, either $x_i$ or $y_i$ belongs to $\mathfrak{p}_i$. If $x_i\in\mathfrak{p}_i$ then $x_{i-1}=f_i(x_i)\in\mathfrak{p}_{i-1}$ and similarly for $y_i$. Thus for arbitrary large $N$, we have either $x_i\in\mathfrak{p}_i$ or $y_i\in\mathfrak{p}_i$ for all $i\leq N$. Now the lemma follows since if $x\notin\mathfrak{p}$, there exists $n$ such that $x_i\notin\mathfrak{p}_i$ for all $i\geq n$, but then $y_i\in\mathfrak{p}_i$ for all $i\geq n$ and hence for all $i$. This proves $y\in\mathfrak{p}$.
\end{proof}

\begin{proposition}
    There exist one-to-one correspondences:
    \begin{align*}
        \mathfrak{K} &\leftrightarrow \spm K^{\circ\flat}_\Delta \\
        \text{and} \quad \mathfrak{S} &\leftrightarrow \spec K^{\circ\flat}_\Delta \setminus \spm K^{\circ\flat}_\Delta.
    \end{align*} 
\end{proposition}

\begin{proof}
    By Lemma \ref{lemma2.3}, all closed prime ideals of $K^{\circ\flat}_\Delta = \varprojlim_{x \mapsto x^p} K^\circ_\Delta / I_\varpi$ can be written as the inverse limit of prime ideals. Note that maximal ideals are, by definition, equal to their closures, so they are automatically closed. The maximal ideals of $K^{\circ}_\Delta / I_\varpi$ are precisely of the form $\mathfrak{m}_s / I_\varpi$, and they are compatible with respect to the Frobenius map in the sense of Lemma \ref{lemma2.3}. Indeed, if $x^p \in \mathfrak{m}_s$, then $\overline{\beta}_s(x)^p = \overline{\beta}_s(x^p) = 0$, which implies $\overline{\beta}_s(x) = 0$ and thus $x \in \mathfrak{m}_s$. For every $s \in \mathfrak{K}$, we have the following commutative diagram:
    \[
    \begin{tikzcd}
        \cdots \arrow[r] & k_s \arrow[r, "x \mapsto x^p"] & k_s \arrow[r, "x \mapsto x^p"] & k_s \\
        \cdots \arrow[r] & K^\circ_\Delta / I_\varpi \arrow[r, "x \mapsto x^p"] \arrow[u] & K^\circ_\Delta / I_\varpi \arrow[r, "x \mapsto x^p"] \arrow[u] & K^\circ_\Delta / I_\varpi \arrow[u] \\
        \cdots \arrow[r] & \mathfrak{m}_s / I_\varpi \arrow[r, "x \mapsto x^p"] \arrow[u, hookrightarrow] & \mathfrak{m}_s / I_\varpi \arrow[u, hookrightarrow] \arrow[r, "x \mapsto x^p"] & \mathfrak{m}_s / I_\varpi \arrow[u, hookrightarrow]
    \end{tikzcd}.
    \]
    Applying the inverse limit functor to the short exact sequence of $K^\circ_\Delta$-modules $0 \to \mathfrak{m}_s / I_\varpi \to K^\circ_\Delta / I_\varpi \to k_s \to 0$, we obtain the exact sequence:
    \begin{align*}
       0 \to \varprojlim_{x \mapsto x^p} \mathfrak{m}_s / I_\varpi \to \varprojlim_{x \mapsto x^p} K^\circ_\Delta / I_\varpi \to \varprojlim_{x \mapsto x^p} k_s \to \varprojlim\nolimits^{(1)} \mathfrak{m}_s / I_\varpi.
    \end{align*}
    We claim that $\Phi: \mathfrak{m}_s / I_\varpi \to \mathfrak{m}_s / I_\varpi$ is surjective. To prove this, let $x \in \mathfrak{m}_s$. Since the Frobenius map $\Phi: K^\circ_\Delta / I_\varpi \to K^\circ_\Delta / I_\varpi$ is surjective, there exists an element $y \in K^\circ_\Delta$ such that $y^p = x + a$ for some $a \in I_\varpi$. Since $I_\varpi \subset \bigcap_{s \in \mathfrak{K}} \mathfrak{m}_s$, it follows that $y^p \in \mathfrak{m}_s$. This is equivalent to $\overline{\beta}_s(y^p) = \overline{\beta}_s(y)^p = 0$ in $k_s$, and hence $y \in \mathfrak{m}_s$. The fact that $\Phi(y + I_\varpi) = y^p + I_\varpi = x + I_\varpi$ confirms surjectivity.

    Therefore, $\mathfrak{m}^\flat_s = \varprojlim_{x \mapsto x^p} \mathfrak{m}_s / I_\varpi$ fits into the canonical exact sequence of $K^\circ_\Delta$-modules:
    \begin{align*}
        0 \to \mathfrak{m}^\flat_s \to K^{\circ\flat}_\Delta \to k_s \to 0.
    \end{align*}
    One can verify that the second map is a ring homomorphism; thus, $\mathfrak{m}^\flat_s$ is a maximal ideal of $K^{\circ\flat}_\Delta$ with residue field $k_s$.

    Now, let $\mathfrak{p}^\flat$ be a prime ideal of $K^{\circ\flat}_\Delta$. For all $\alpha \in \Delta$, we have the homomorphism:
    \begin{align*}
        \iota^\flat_\alpha: K^{\circ\flat}_\alpha &\to K^{\circ\flat}_\Delta, \\
        (x^{(n)})_{n \in \mathbb{N}} &\mapsto (\iota_\alpha(x^{(n)}))_{n \in \mathbb{N}}.
    \end{align*}
    Consider the composition $\phi^\flat_\alpha: K^{\circ\flat}_\alpha \to K^{\circ\flat}_\Delta \to K^{\circ\flat}_\Delta / \mathfrak{p}^\flat$. The kernel $\ker \phi^\flat_\alpha$ is a prime ideal of $K^{\circ\flat}_\alpha$, which must be either $\mathfrak{m}_{K^\flat_\alpha}$ or $(0)$. Suppose that for some $\alpha_0 \in \Delta$, $\ker \phi^\flat_{\alpha_0} = \mathfrak{m}_{K^\flat_{\alpha_0}}$. Since $\mathfrak{m}_{K^\flat_{\alpha_0}}$ contains $\pi_{\alpha_0}$, we have $\pi_{\alpha_0} \in \mathfrak{p}^\flat$. The projections to $K^\circ / I_\varpi$ are prime ideals of $K^\circ_\Delta$ containing $p$. Since the only prime ideals containing $p$ are of the form $\mathfrak{m}^\flat_s$ for some $s \in \mathfrak{S}$, the previous argument implies $\mathfrak{p}^\flat = \mathfrak{m}^\flat_s$. Consequently, all homomorphisms $\phi_\alpha$ can be assumed to be injective. Therefore, for every $\alpha\in\Delta$, we get an embeddings $\phi^\flat_\alpha: K^{\circ\flat}_\alpha\to \widehat{\phi_1(K^{\circ\flat}_1)\cdots\phi_t(K^{\circ\flat}_t)}\subset\C^\flat$. Untilting, we obtain a family of embeddings $s=(\phi_\alpha: K^\circ_\alpha\to \C)_{\alpha\in\Delta}$. Also note that the diagram
    \[
    \begin{tikzcd}
        K^\flat_s\arrow[rr,"u^\flat"]&  &K'^\flat_{s'}\\
                     & \arrow[ul,"\phi^\flat_\alpha"]K^\flat_\alpha\arrow[ur,swap, "\phi'^\flat_\alpha"] &  
    \end{tikzcd}
    \]
    commutes if and only if 
    \[
    \begin{tikzcd}
        K_s\arrow[rr,"u"]&  &K'_{s'}\\
                     & \arrow[ul,"\phi_\alpha"]K_\alpha\arrow[ur,swap,"\phi'_\alpha"] &  
    \end{tikzcd}
    \]
    commutes. Therefore, we get an injective map $\spec K^{\circ\flat}_\Delta\to\mathfrak{S}$. But we know that for every $s\in\mathfrak{S}$, we have a prime ideal $\mathfrak{p}^\flat_s$ which is the support of valuation $|\bullet|_{\flat,s}$, we get the bijection. 
\end{proof}

\begin{remark}
    From Lemma \ref{lemma2.3}(iii), it is easy to see that $\mathfrak{p}^\flat_s, s\in\mathfrak{S}$ cannot be written as the inverse limit of ideals in $K^\circ_\Delta/I_\varpi$. This is because only prime ideals of $K^\circ_\Delta$ containing $I_\varpi$ are $\mathfrak{m}_s$. But if we assume the proposition, it is easy to see that $\mathfrak{p}^\flat_s$ is $\varprojlim\limits_{x\mapsto x^p}\mathfrak{p}_s$. The inverse limit makes sense, since Frobenius maps elements of $\mathfrak{p}_s$ to elements of $\mathfrak{p}_s$.
\end{remark}

\begin{remark}
	We denote the ring homomorphisms obtained above by
	\begin{align*}
		\beta^\flat_s: K^{\circ\flat}_\Delta\to K^\flat_s
	\end{align*}
	and 
	\begin{align*}
		\overline{\beta}^\flat_s: K^{\flat\circ}_\Delta\to k_s
	\end{align*}
	whose kernels are $\ker\beta^\flat_s=\mathfrak{p}^\flat_s$ and $\ker\overline{\beta}^\flat_s=\mathfrak{m}^\flat_s$. 
\end{remark}

\begin{corollary}
    Every pseudo-uniformizing ideal $I_\varpi$ ( resp. $I^\flat_\varpi$) is principal.
\end{corollary}

\begin{proof}
   First, suppose $I_\varpi$ is generated by $\varpi_1=\varpi\otimes1\otimes\cdots\otimes 1,\ldots,\varpi_t=1\otimes 1\otimes\cdots \otimes \varpi$. We want to show that for every $i\neq j$, $\varpi_i$ and $\varpi_j$ differ by a unit in $K^\circ_\Delta$. Indeed, for every $s\in \mathfrak{S}$, $\left|\frac{\varpi_i}{\varpi_j}\right|_s=\left|\frac{\beta_s(\varpi)}{\beta_s(\varpi)}\right|=1$. Thus, this element does not lie in $\mathfrak{m}_s$ for any $s\in\mathfrak{S}$, making it a unit. Therefore, $\varpi_i=u\varpi_j$ for some $u\in (K^\circ_\Delta)^\times$. For $I_\varpi$ constructed for arbitrary $\varpi_\alpha$, we can replace it with any one of the above, which reduces it to the case of valuation rings. A similar proof works for $K^{\circ\flat}_\Delta$.
\end{proof}

From now on we will write $I_\varpi=(\varpi)$ for a pseudo-uniformizer $\varpi$ and $I^\flat_\varpi=(\varpi^\flat)$.

Let us end this section with the following definition, which will be important throughout the article
\begin{definition}
    A \emph{perfectoid $\Delta$-field} of mixed characteristic is defined as a ring isomorphic to $K_\Delta$, as constructed above.
\end{definition}

\begin{remark}
     The term field may be misleading, as the ring $K_\Delta$ is not, in general, a field. In fact, it is not even an integral domain. However, it is worth noting that $K_\Delta$ is reduced and has Krull dimension 0. As a result, every element of $K_\Delta$ is either a unit or a zero divisor, making it a total ring of fractions of $K^\circ_\Delta$. We use the term field here because we are replacing a perfectoid field with $K_\Delta$, and the translation of classical results carries over smoothly in this context
\end{remark}


\section{Objects in equicharacteristics setups}

In this section, we compare $K^{\circ\flat}_\Delta$ with the rings $K^{\circ\flat}_\alpha$ for $\alpha \in \Delta$. Note that the approach from Section~\S\ref{section1} does not apply here. In particular, $K^{\circ\flat}_\Delta$ may not be isomorphic to the completed tensor product $K^{\circ\flat}_1 \widehat{\otimes}_{K^{\circ\flat}_0} \cdots \widehat{\otimes}_{K^{\circ\flat}_0} K^{\circ\flat}_t$ for any base field $K^\flat_0$ such that $K^\flat_0 \subset K^\flat_\alpha \subset \C^\flat$ for all $\alpha\in\Delta$. This is because, in general, the latter ring is not reduced.

\begin{definition}
   A \emph{perfectoid $\Delta$-field} of equicharacteristic is defined as a ring $K^\flat_\Delta$ that is isomorphic to the total ring of fractions of $K^{\circ\flat}_\Delta$, as constructed above. A \emph{perfectoid $\Delta$-field} is a ring that is either a perfectoid $\Delta$-field of mixed characteristic or one of equicharacteristic.
\end{definition}

Note that $K^{\circ\flat}_\Delta\left[\frac{1}{\varpi^\flat}\right]=K^\flat_\Delta$ for the uniformizer $\varpi$ of $K_0$. 
Throughout this article, to maintain consistency with standard notation, we will use $\varpi$ to denote $\varpi^\flat$ and $I_\varpi$ for $I^\flat_\varpi$, including for the perfectoid $\Delta$-field of mixed characteristic. 
This convention is adopted because we intend to tilt the perfectoid $\Delta$-field of equal characteristic as part of the same theory. 
Furthermore, we choose $\varpi$ such that $(\varpi^\flat)^\#=\varpi$, ensuring that all $p^n$-th roots of $\varpi$ exist. 
One can easily verify that $K^{\circ\flat}_\Delta$ is the set of power-bounded elements of $K^\flat_\Delta$; hence, we also write $K^{\flat\circ}_\Delta$ for $K^{\circ\flat}_\Delta$.

\begin{proposition}
    There is a one to one correspondence between $\mathfrak{S}$ and $\spec K^\flat_\Delta$
\end{proposition}

\begin{proof}
    Since $K^\flat_\Delta$ is the total ring of quotients of $K^{\circ\flat}_\Delta$, this is the localization with respect to the multiplicative set of non-zero divisors. Therefore, the prime ideals of this localization are the localization of prime ideals of $K^{\circ\flat}_\Delta$ which are disjoint from the multiplicative subset. Thus, only prime ideals that matter to us must have all their elements as zero divisors, which are $\mathfrak{p}^\flat_s$ for $s\in\mathfrak{S}$, and hence we get a prime ideal  $\mathfrak{P}^\flat_s$ for every $s\in\mathfrak{S}$ with residue field $K^\flat_s$.
\end{proof}

\begin{proposition}
	There is a isomorphism between topological rings $K^\flat_\Delta$ and $\varprojlim\limits_{x\mapsto x^p} K_\Delta$
\end{proposition}

\begin{proof}
     Note that we have an injective ring homomorphism
     \begin{align*}
     	K^{\circ\flat}_\Delta=\varprojlim\limits_{x\mapsto x^p}K^{\circ}_\Delta\to \varprojlim\limits_{x\mapsto x^p}K_\Delta
     \end{align*}
     This gives an injective ring map to the total ring of fractions
     \begin{align*}
     	K^\flat_\Delta\to Q(\varprojlim\limits_{x\mapsto x^p}K_\Delta)
     \end{align*}
     This map is clearly surjective as for the uniformizer $\varpi\in K_0$, we have a non-zero divisor $\varpi'$ made from compatible $p^n$-th roots of $\varpi$, then $\left(\varprojlim\limits_{\Phi}K^\circ_\Delta\right)\left[\frac{1}{\varpi}\right]=\varprojlim\limits_{\Phi} K^\circ_{\Delta}\left[\frac{1}{\varpi^{1/p^m}}\right]=\varprojlim\limits_\Phi K_\Delta$.
     Thus, it is enough to show every element of $\varprojlim\limits_{x\mapsto x^p}K_\Delta$ is either a unit or a zero divisor.  But we know that every prime ideal of $K^{\circ\flat}_\Delta$ are of the form $\mathfrak{p}^\flat_s:=\varprojlim\limits_{x\to x^p}\mathfrak{p}_s$ for $s\in\mathfrak{S}$, thus we have every prime ideal of $K^\circ_\Delta$ is of the form $\mathfrak{P}^\flat=\varprojlim\limits_{x\to x^p}\mathfrak{P}_s$. This is also a reduced ring, and each prime ideal is maximal; every element is a unit or a zero divisor. Hence is the total ring of fractions.
\end{proof}

Let $K_0$ be a p-adic field and $K_\alpha, \alpha\in\Delta$ be perfectoid fields such that $K_0\subset K_\alpha\subset \C$ for all $\alpha\in\Delta$. Recall that $\mathfrak{S}$ is the equivalence classes of the embedding family defined above. Tilting the perfectoid fields $K_\alpha,K_s, \alpha\in\Delta,s\in\mathfrak{S}$, we obtain the analogous situation in characteristics $p$ i.e. $K^\flat_\alpha, K^\flat_s\subset \C^\flat$ for all $\alpha\in\Delta$ and $s\in\mathfrak{S}$. Using the theorem of Fontaine-Wintenberger \cite{FonWin}, we have isomorphism $G_{K_\alpha}\cong G_{K^\flat_\alpha}$ for all $\alpha\in\Delta$ and $G_{K_s}\cong G_{K^\flat_s}$ for all $s\in\mathfrak{S}$. Consider a non-Archimedean field $K^\flat_0$ of characteristic $p$ such that $K^\flat_0\subset K^\flat_\alpha$ for all $\alpha\in\Delta$ and the set of equivalence classes of embedding families defined analogously for $K^\flat_0\subset K^\flat_\alpha\subset \C^\flat$ is in bijection with $\mathfrak{S}$. Consider $A=K^\flat_1\hat{\otimes}_{K^\flat_0}\cdots\hat{\otimes}_{K^\flat_0} K^\flat_t$ and we have the following proposition:

\begin{proposition}
    The ring $K^\flat_\Delta$ is the reduced ring corresponding to $A$, that is
    \begin{align*}
        K^\flat_\Delta\cong A/\Nil A
    \end{align*}
    where $\Nil A$ is the nilradical of $A$.
\end{proposition}

\begin{proof}
    Consider the surjective ring homomorphism
    \begin{align*}
        \theta: A&\to K^\flat_\Delta\\
        \sum\limits_{i=0}^\infty (x^{(n)}_{1i})_n\otimes\cdots\otimes (x^{(n)}_{ti})_n&\mapsto \sum\limits_{i=1}^\infty (x^{(n)}_{1i}\otimes\cdots\otimes x^{(n)}_{ti})_n
    \end{align*}
    and by assumption all prime ideals of $A$ are $\mathfrak{Q}_s=\ker\gamma_s$ with embeddings $(\gamma_\alpha: K^\flat_\alpha\to K_s)_{\alpha\in\Delta}$ for $s\in\mathfrak{S}$. We can define the bijection such that $\theta^{-1}(\mathfrak{P}_s)=\mathfrak{Q}_s$ because the inverse image of a prime ideal is a prime ideal. Let $x\in\ker\theta$, we have $\theta(x)=0\in\bigcap\limits_{s\in\mathfrak{S}}\mathfrak{P}_s$. Thus 
    \begin{align*}
        x\in \theta^{-1}\left(\bigcap\limits_{s\in\mathfrak{S}}\mathfrak{P}_s\right)=\bigcap\limits_{s\in\mathfrak{S}}\mathfrak{Q}_s.
    \end{align*}
    But the intersection of all prime ideals is the nilradical of $A$, and hence we get the desired isomorphism.
\end{proof}

\begin{remark}
   A similar result appears in Proposition 4.3 of \cite{Bri}, where the author consider the $p$-adic completion of the ring $\mathcal{O}_{\C} \otimes_{W(k_0)} \cdots \otimes_{W(k_0)} \mathcal{O}_{\C}$ and the $I^\flat_p$-completion of $\mathcal{O}^\flat_{\C} \otimes_{k_0} \cdots \otimes_{k_0} \mathcal{O}^\flat_{\C}$. However, we will not be concerned with the existence of $K^{\flat}_0$ in this article. The explicit structure of the Tilt will not play a role in this article.
\end{remark}

The following example shows that tilts of $\Delta$-perfectoid fields are not immediately recoverable from tilts of the individual perfectoid fields. This means that the corresponding Fargues-Fontaine curves, which parametrize the untilts, are essentially different.

\begin{example}
Consider $K_1 = \widehat{\mathbb{Q}_p(p^{1/p^\infty})}$ and $K_2 = \widehat{\mathbb{Q}_p(\mu_{p^\infty})}$ ($p \neq 2$), both of which are perfectoid fields over $K_0 = \mathbb{Q}_p$. Take $K_\Delta = K_1 \widehat{\otimes}_{\mathbb{Q}_p} K_2 = \widehat{\mathbb{Q}_p(p^{1/p^\infty}, \mu_{p^\infty})}$, as they are linearly disjoint. In contrast, consider $K'_\Delta = K_1 \widehat{\otimes}_{\mathbb{Q}_p} K_1 = \widehat{\mathbb{Q}_p(p^{1/p^\infty})} \widehat{\otimes}_{\mathbb{Q}_p} \widehat{\mathbb{Q}_p(p^{1/p^\infty})}$, which contains zero divisors such as $p^{1/p} \otimes 1 - 1 \otimes p^{1/p}$.

However, note that both $K_1$ and $K_2$ have isomorphic tilts: $K_1^\flat = \mathbb{F}_p[[t_1^{1/p^\infty}]]\left[\frac{1}{t_1}\right]$ (by identifying $(p, p^{1/p}, p^{1/p^2}, \dots)$ with $t_1$) and $K_2^\flat = \mathbb{F}_p[[t_2^{1/p^\infty}]]\left[\frac{1}{t_2}\right]$ (by identifying $(1, \zeta_p, \zeta_{p^2}, \dots)-1$ with $t_2$).

Therefore, the tilts of perfectoid $\Delta-$fields and perfectoid fields are essentially different. If we want to prove that $K_\Delta^\flat = \mathbb{F}_p[[t_1^{1/p^\infty}]]\left[\frac{1}{t_1}\right] \widehat{\otimes}_{\mathbb{F}_p} \mathbb{F}_p[[t_2^{1/p^\infty}]]\left[\frac{1}{t_2}\right]$, as in \cite{Bri}, then we need to see $t_1$ and $t_2$ inside some bigger field.

Note that since $K_\Delta$ is a field, a map exists from
\begin{align*}
\mathbb{F}_p[[t_1^{1/p^\infty}]]\left[\frac{1}{t_1}\right] \widehat{\otimes}_{\mathbb{F}_p} \mathbb{F}_p[[t_2^{1/p^\infty}]]\left[\frac{1}{t_2}\right] = \mathbb{F}_p[[t_1^{1/p^\infty}, t_2^{1/p^\infty}]]\left[\frac{1}{t_1}, \frac{1}{t_2}\right] \to K_\Delta^\flat.
\end{align*}
This ring is considered in \cite{Zab, CarKedZab, Pal-Zab}, but it cannot be the tilt because it has rank-2 valuations (if $t_1$ and $t_2$ are not contained inside a perfectoid field), which is not possible for a perfectoid field. In the latter case, we have a surjective map
\begin{align*}
\mathbb{F}_p[[t^{1/p^\infty}_1]]\left[\frac{1}{t_1}\right]\hat{\otimes}_{\mathbb{F}_p}\mathbb{F}_p[[t^{1/p^\infty}_2]]\left[\frac{1}{t_2}\right]\to K'^{\flat}_\Delta.
\end{align*}
\end{example}


\section{Non-Archimedean functional analysis over $K_\Delta$}

Throughout the section, we fix a perfectoid $\Delta-$field $K_\Delta$. We recall the following definition from the classical ring theory \cite{Goode}.

\begin{definition}
    A commutative ring $R$ with 1 is called a \emph{von Neumann regular (VNR)} ring if for every $x\in R$ there exists $y\in R$ such that $x^2y=x$.
\end{definition}

\begin{proposition}
    The following conditions are equivalent for a commutative ring $R$ with $1$:
    \begin{enumerate}
        \item[(i).] $R$ is VNR;
        \item[(ii).] Every principal ideal of $R$ is generated by an idempotent element;
        \item[(iii).] Every finitely generated ideal of $R$ is generated by an idempotent;
        \item[(iv).] Every principal ideal is a direct summand of $R$;
        \item[(v).] Every finitely generated submodule of a projective $R-$module $P$ is a direct summand of $P$;
        \item[(vi).] $R$ is absolutely flat, i.e., every module over $R$ is flat;
        \item[(vii).] $R$ is a reduced ring of dimension $0$;
        \item[(viii).] Every localization of $R$ at a maximal ideal is a field.
    \end{enumerate}
\end{proposition}

\begin{proof}
    See \cite{Goode}.
\end{proof}

Since the ring $K_\Delta$ is a reduced ring of dimension $0$, it is a VNR ring. Therefore, every module over $K_\Delta$ is a flat module. 

\begin{lemma}
  For every $s\in\mathfrak{S}$, the localization $(K_\Delta)_{\mathfrak{P}_s}$ is isomorphic to $K_s$.
\end{lemma}

\begin{proof}
    For each $s\in\mathfrak{S}$, we have a well defined map
    \begin{align*}
        (K_\Delta)_{\mathfrak{P}_s}&\to K_s\\
                        \frac{a}{b}&\mapsto \frac{\beta_s(a)}{\beta_s(b)}
    \end{align*}
    Indeed, if $b\not\in\mathfrak{P}_s$, then $\beta_s(b)\neq 0$ and hence above map is well defined. It is also surjective. Since the ring $K_\Delta$ is VNR, this map is also injective as $(K_\Delta)_{\mathfrak{P}_s}$ is a field.
\end{proof}

\begin{lemma}
    The localizations $(K^\circ_\Delta)_{\mathfrak{p}_s}$ is isomorphic to $K_s$ for all $s\in\mathfrak{S}$ and $(K^\circ_\Delta)_{\mathfrak{m}_s}$ is isomorphic to $K^\circ_s$ for all $s\in\mathfrak{K}$.
\end{lemma}

\begin{proof}
   For $s\in\mathfrak{S}$, consider the well-defined map:
\begin{align*}
  \phi_s: (K^\circ_\Delta)_{\mathfrak{p}_s}&\to K_s\\
  \frac{a}{b}&\mapsto \frac{\beta_s(a)}{\beta_s(b)}
\end{align*}
Note that $\ker\phi_s=\Big\{\frac{a}{b}\in (K^\circ_\Delta)_{\mathfrak{p}_s}\mid\beta_s(a)=0\Big\}=\Big\{\frac{a}{b}\mid a\in\mathfrak{p}_s \text{ and }b\not\in\mathfrak{p}_s\Big\}=(\mathfrak{p}_s)_{\mathfrak{p}_s}$. However, since $\mathfrak{p}_s=\mathfrak{P}_s\cap K^\circ_\Delta$ for all $s\in\mathfrak{S}$, we have $(\mathfrak{p}_s)_{\mathfrak{p}_s}=0$. Hence, the above homomorphism is injective. It is also surjective, and therefore an isomorphism.

Now, for $(K^\circ_\Delta)_{\mathfrak{m}_s}$, consider the homomorphism:
\begin{align*}
  \Phi_s:(K^\circ_\Delta)_{\mathfrak{m}_s}&\to K^\circ_s\\
  \frac{a}{b}&\mapsto \frac{\beta_s(a)}{\beta_s(b)}
\end{align*}
This map is clearly surjective. Let $\frac{a}{b}\in\ker\Phi_s$. Then $a\in\ker\beta_s=\mathfrak{p}_s$, which implies $a$ is a zero divisor. There exists $x\not\in\mathfrak{p}_s$ such that $ax=0$, as this is precisely what the statement $(\mathfrak{p}_s)_{\mathfrak{p}_s}=0$ translates to. Therefore, $\ker\Phi_s=0$.
\end{proof}

Supoose $M$ is a $K_\Delta-$module. Since $K_\Delta$ is a VNR ring, $M$ is a flat $K_\Delta-$module. Therefore, from the short exact sequence
\[
\begin{tikzcd}
    0\arrow[r] & \mathfrak{P}_s\arrow[r]&K_\Delta\arrow[r,"\beta_s"] & K_s\arrow[r] & 0,
\end{tikzcd}
\]
we obtain a short exact sequence
\[
\begin{tikzcd}
    0\arrow[r]& \mathfrak{P}_s\otimes_{K_\Delta} M\arrow[r] & M\cong K_\Delta\otimes_{K_\Delta} M\arrow[r]& M_s=M\otimes_{K_\Delta} K_s\arrow[r]&0.
\end{tikzcd}
\]
Also since $(K_\Delta)_{\mathfrak{P_s}}\cong K_s$, we must have $M_{\mathfrak{P}_s}:=M\otimes_{K_\Delta} (K_\Delta)_{\mathfrak{P}_s}\cong M\otimes_{K_\Delta}K_s=:M_s$.\\

Similarly, if $N$ is a $K^\circ_\Delta-$module, then $N_{\mathfrak{p}_s}:=N\otimes_{K^\circ_\Delta}(K^\circ_\Delta)_{\mathfrak{p}_s}\cong N\otimes_{K^\circ_\Delta}K_s$ and $N_{\mathfrak{m}_s}:=N\otimes_{K^\circ_\Delta}(K^\circ_\Delta)_{\mathfrak{m}_s}\cong N\otimes_{K^\circ_\Delta} K^\circ_s$. Therefore, by the local criterion, the flatness gives the following proposition.

\begin{proposition}[\label{Flatness criterion}]
	A module $M$ over $K^\circ_\Delta$ is flat if and only if $M$ is torsion-free.
\end{proposition}

\begin{proof}
 We know that a module $M$ over $K^\circ_\Delta$ is flat if and only if the $(K^\circ_\Delta)_{\mathfrak{m}_s}$-module $M_{\mathfrak{m}_s}$ is flat for every $s \in \mathfrak{K}$. Since $K_\Delta$ is the total ring of fractions of $K^\circ_\Delta$, which is absolutely flat, it follows that being torsion-free is a local property. Therefore, using the flatness criterion for the valuation ring $(K^\circ_\Delta)_{\mathfrak{m}_s}$, we obtain the result.
\end{proof}

\begin{definition}
	A \emph{normed space over $K_\Delta$} is a $K_\Delta$-module $M$ equipped with a family of maps, called a family of norms,
\begin{align*}
    \|\bullet\|_s: M &\to \mathbb{R}_{\geq 0}
\end{align*}
indexed by $s \in \mathfrak{S}$, satisfying the following identities: for all $x,y\in M$ and $\lambda\in K_\Delta$, we have
\begin{enumerate}
    \item[(i)] $\|x + y\|_s \leq \max\{\|x\|_s, \|y\|_s\}$,
    \item[(ii)] $\|\lambda x\|_s = |\lambda|_s \|x\|_s$, where $|\bullet|_s=|\beta_s(\bullet)|$,
    \item[(iii)] $\bigcap\limits_{s \in \mathfrak{S}} \ker \|\bullet\|_s = \{0\}$, where $\ker \|\bullet\|_s := \{x \in M \mid \|x\|_s = 0\}$.
\end{enumerate}
We define the \emph{norm} of $M$ as $\|x\| = \sup\limits_{s \in \mathfrak{S}} \|x\|_s$. This defines a metric
\begin{align*}
    d(x, y) = \|x - y\|
\end{align*}
which induces a topology on $M$.

Note that if $\|x\| = \sup\limits_{s \in \mathfrak{S}} \|x\|_s = 0$, then $\|x\|_s = 0$ for all $s \in \mathfrak{S}$. By (iii), it follows that $x = 0$. A \emph{morphism of normed spaces} is a continuous $K_\Delta$-module homomorphism.

A \emph{Banach space} over $K_\Delta$ is a norm space $M$ that is complete with respect to the above metric.
\end{definition}

\begin{theorem}{\emph{(Open Mapping Theorem)}}\label{Open Mapping Theorem}
	Let $T:E\to F$ be a surjective morphism of Banach spaces. Then $T$ is an open mapping, i.e., if $U$ is open in $E$, then $T(U)\subset F$ is also open.
\end{theorem}

\begin{proof}
	Let $\varpi$ be a pseudo-uniformizer of $K_\Delta$ satisfying $|\varpi|_s=|\varpi|_t$ for all $s,t\in\mathfrak{S}$.  Set $\delta=|\varpi|$ and note that for every $n\in\mathbb{Z}$,
	\begin{align*}
		\B_E(0,\delta^n)&:=\{x\in E\mid \|x-0\|<\delta^n\}\\
		                        &=\{x\in E\mid |\varpi|^{-n}\|x\|<1\}\\
		                        &=\{x\in E\mid \|\varpi^{-n}x\|<1\}\\
		                        &=\varpi^n\B_E(0,1).
	\end{align*}
	Also write $\B_E:=\B_E(0,1)$ which is a subgroup satisfying $\bigcup\limits_{n=1}^\infty\B_E\left(0,\frac{1}{\delta^n}\right)=\bigcup\limits_{n=1}^\infty\varpi^{-n}\B_E=E$. Therefore, we have
	\begin{align*}
		F=\bigcup\limits_{n=1}^\infty \overline{T(\varpi^{-n}\B_E)}=\bigcup\limits_{n=1}^\infty \varpi^{-n}\overline{T(\B_E)}.
	\end{align*}
	By Baire's Category Theorem, there exists $k\in\mathbb{N}^*$ such that $\varpi^{-k}\overline{T(\B_E)}$ has an interior point and hence $\overline{T(\B_E)}$ has an interior point ($x\mapsto \varpi^{-k}x$ is a homeomorphism of $F$). Thus there exists $g\in\B_E$ such that $Tg\in\left(\overline{T(\B_E)}\right)^\circ$, interior of $\overline{T(\B_E)}$. As $\B_E$ is closed under addition, we get
	\begin{align*}
		0\in\left(\overline{T(\B_E-g)}\right)^\circ=\left(\overline{T(\B_E)}\right)^\circ.
	\end{align*}
	Now note that $\{\varpi^n\B_F\}_{n\geq 1}$ form a fundamental system of neighborhoods of $0$ in $F$ because $\lim\limits_{n\to\infty}\varpi^n=0$. There exists $m\geq 1$ such that $\varpi^m\B_F\subset \overline{T(\B_E)}$. By the definition of closure, we have
	\begin{align*}
		h\in\varpi^m\B_F\text{ and }\epsilon>0\implies \text{ there exists } f\in\B_E\text{ such that }\|h-Tf\|<\epsilon.
	\end{align*}
	Let $h$ be arbitrary element of $F$, then there exists $n\in\mathbb{Z}$ such that $|\varpi|^n\leq \|h\|<|\varpi|^{n+1}$ if $n<0$ or $|\varpi|^{n+1}\leq\|h\|<|\varpi|^n$ if $n\geq 0$. Consider
	\begin{align*}
		h'=
		\begin{cases}
			\varpi^{m-n}h &\text{ if } n\geq 0\\
			\varpi^{m-n-1}h &\text{ if } n<0
		\end{cases}.
	\end{align*}
	Note that $h'\in\varpi^n\B_F$ and hence applying above to $h'$, we get:
	\begin{equation}\label{eq 1}
		h\in F \text{ and }\epsilon>0\implies \text{ there exists } f\in
		\begin{cases}
			\varpi^{n-m}\B_E &\text{ if } n\geq 0\\
			\varpi^{n+1-m}\B_E &\text{ if }n<0
		\end{cases}
		\text{ such that }\|h-Tf\|<\epsilon.
	\end{equation}
	Suppose $g\in F$ and $\|g\|<1$. Applying equation (\ref{eq 1}) with $h=g$ and $\epsilon=\delta$ we get:
	\begin{align*}
		\text{ there exists } f_1\in \varpi^m\B_E \text{ such that }\|g-Tf_1\|<\delta.
	\end{align*}
	Again applying equation (\ref{eq 1}) with $h=g-Tf_1$ and $\epsilon=\delta^2$, we get
	\begin{align*}
		\text{there exists }f_2\in\varpi^{m+1}\B_E \text{ such that }\|g-Tf_1-Tf_2\|<\delta^2.
	\end{align*}
	Continuing this process, we can construct a sequence $f_1,f_2,\ldots$ in $E$ such that $f_k\in\varpi^{m+k-1}\B_E$ and $\|g-Tf_1-Tf_2-\cdots-Tf_k\|<\delta^n$. Let $f=\sum\limits_{k=1}^\infty f_k\in \varpi^m\B_E$ which converges because $f_k\to 0$. Indeed, $\varpi^{m+k}\B_E\subset\varpi^m\B_E$, which is a subgroup. Since $T$ is continuous, we obtain $g=Tf$. Therefore, $\B_F\subset T\left (\varpi^{m}\B_E\right)$ which proves the result.
\end{proof}

\begin{definition}
    A \emph{Banach algebra over $K_\Delta$} is a $K_\Delta-$algebra $R$ together with a family of norms $\|\bullet\|_s,s\in\mathfrak{S}$ such that:
    \begin{enumerate}
        \item[(i)] $(R,(\|\bullet\|_s: s\in\mathfrak{S}))$ is a Banach space over $K_\Delta$;
        \item[(ii)] For all $x,y\in R$ and for all $s\in\mathfrak{S}$, we have $\|xy\|_s\leq \|x\|_s\|y\|_s$;
        \item[(iii)] For all $s\in\mathfrak{S}$, $\|1\|_s=1$ where $1$ is the unit element.
    \end{enumerate}
\end{definition}

\begin{example}
  Recall from \cite{Hubet} that for a Tate ring $A$, we define
\begin{align*}
    A\left\langle T \right\rangle = A\left\langle T_1, \dots, T_n \right\rangle := \left\{ \sum\limits_{\nu \in \mathbb{N}^n} a_\nu T^\nu \,\middle|\, a_\nu \to 0 \text{ as } |\nu| \to \infty \right\},
\end{align*}
equipped with the topology for which a fundamental system of neighborhoods of 0 is given by
\begin{align*}
    U\left\langle T \right\rangle = U\left\langle T_1, \dots, T_n \right\rangle := \left\{ \sum\limits_{\nu \in \mathbb{N}^n} a_\nu T^\nu \,\middle|\, a_\nu \in U \text{ for every } \nu \in \mathbb{N}^n \right\}
\end{align*}
where $U$ runs over the fundamental system of neighborhoods of 0 on $A$.
For $A = K_\Delta$, we have a natural family of norms on $K_\Delta\left\langle X \right\rangle$; for $f = \sum\limits_{\nu \in \mathbb{N}^n} a_\nu X^\nu$, we define
\begin{align*}
    \|f\|_s := \sup\{ |a_\nu|_s \mid \nu \in \mathbb{N}^n \},
\end{align*}
with respect to which this form a Banach algebra. We call this the \emph{Tate algebra over $K_\Delta$}, and refer to the corresponding family of norms as the \emph{Gauss norm family}. The topology induced by the norm associated with this Banach algebra coincides with the topology defined above.
\end{example}

\begin{example}
    The $K_\Delta-$module $K_s$ is also a Banach algebra with the family of norms
    \begin{align*}
        \|x\|_{s'}:=
        \begin{cases}
            |x|_s & \text{ if } s'=s\\
            0     & \text{ if } s'\neq s
        \end{cases}.
    \end{align*}
\end{example}

\begin{definition}
    A Banach $K_\Delta-$algebra $R$ is called \emph{topologically of finite type} if there exists a surjective $K_\Delta-$algebra homomorphism $K_\Delta\left<T_1,\cdots,T_n\right>\to R$ for some $n\geq 0$.
\end{definition}

\begin{remark}
    It remains uncertain whether every quotient algebra of $K_\Delta\left<T_1,\cdots,T_n\right>$ inherently possesses a Banach $K_\Delta$-algebra structure. The existence of a Banach algebra structure on such quotients is guaranteed if every ideal within $K_\Delta\left<T_1,\cdots,T_n\right>$ is closed.
\end{remark}

Note that the ring $K_\Delta$ is not Noetherian in general. For example, take $K_\Delta = \mathbb{Q}_p(p^{1/p^\infty})\widehat{\otimes}_{\mathbb{Q}_p}\mathbb{Q}_p(p^{1/p^\infty})$, then the chain of principal ideals $(p^{1/p}\otimes 1 - 1\otimes p^{1/p}) \subset (p^{1/p^2}\otimes 1 - 1\otimes p^{1/p^2}) \subset \cdots$ does not stabilizes. But as every finitely generated ideal of $K_\Delta$ is principal, $K_\Delta$ is a coherent ring.


\section{Almost Mathematics in multivariate setups}

Fix a perfectoid $\Delta$-field $K_\Delta$ with maximal ideals $\mathfrak{P}_s$ indexed by the corresponding collection of embedding families $\mathfrak{S}$. For $s\in\mathfrak{S}$, the corresponding residue field is $K_s = K_\Delta / \mathfrak{P}_s$, and the canonical surjection is denoted by $\beta_s: K_\Delta \to K_s$. Define
\[
    \mathfrak{h} := \{x\in K_\Delta\mid |x|<1\}.
\]
This ideal is clearly flat and satisfies $\mathfrak{h}^2 = \mathfrak{h}$. This is also the Jacobson radical of $K^\circ_\Delta$. Note that $K_\Delta$ is Tate with ring of definition $K^\circ_\Delta$ and ideal of definition $I_\varpi=(\varpi)$, we have $K^\circ_\Delta\left[\frac{1}{\varpi}\right]=K_\Delta$. Thus, $(K^\circ_\Delta, \mathfrak{h})$ satisfies the assumptions of Faltings' almost mathematics \cite{GFal}, \cite{ART}. We have the following definition in this case
\begin{definition}
    Let $M$ be a $K^\circ_\Delta-$module. An element $x\in M$ is \emph{almost zero} if $\epsilon x=0$ for every $\epsilon\in\mathfrak{h}$. $M$ is called \emph{almost zero} if every element of $M$ is almost zero. It is equivalent to the condition $\varpi^{1/p^n}M=0$ for all $n$.
\end{definition}

The full subcategory of almost zero modules forms a thick Serre subcategory, and hence we can define the Serre Quotient, which we denote by $K^{\circ a}_\Delta-\Mod$. Thus, we have a localization functor
\begin{align*}
   K^\circ_\Delta-\Mod&\to K^{\circ a}_\Delta-\Mod,\\
                	              M&\mapsto M^a.
\end{align*}
By abstract nonsense, this functor is exact and essentially surjective with kernel the Serre subcategory of almost zero modules. The objects on $K^{\circ a}_\Delta-\Mod$ are objects of $K^\circ_\Delta-\Mod$ and set of morphisms from $M^a$ to $N^a$ in $K^{\circ a}_\Delta$ is given by the following formula for actual $K^\circ_\Delta$-module
\begin{align*}
	\Hom_{K^{\circ a}_\Delta}(M^a, N^a) := \Hom_{K^\circ_\Delta}(\mathfrak{h} \otimes_{K^\circ_\Delta} M, N).
\end{align*}
Tensor product in $K^\circ_\Delta-\Mod$ induces tensor algebra structure in $K^{\circ a}_\Delta-\Mod$. Tensor product is given by the formula; for $M^a,N^a\in K^{\circ a}_\Delta$, we have
\begin{align*}
    M^a\otimes N^a:=(M\otimes_{K^\circ_\Delta} N)^a.
\end{align*}
This tensor product gives an abelian Tensor category structure on $K^{\circ a}_\Delta-\Mod$ and hence one can define algebra objects $A$, i.e., a $K^{\circ a}_\Delta-\Mod$ with a compatible morphism $A\otimes A\to A$. The category of algebra object in $K^{\circ a}_\Delta-\Mod$ is denoted by $K^{\circ a}_\Delta-\Alg$. For every $A\in K^{\circ a}_\Delta-\Alg$, one can abstractly define modules, and this category is denoted by $A-\Mod$. The Serre subcategory is a localizing subcategory in the sense that the localization functor $M\mapsto M^a$ has a right adjoint $K^{\circ a}_\Delta-\Mod\to K^\circ_\Delta-\Mod; M\mapsto M_*$ given by the formula
\begin{align*}
    M_*:=\Hom_{K^{\circ a}_\Delta}(K^{\circ a}_\Delta,M)
\end{align*}
Thus, the abelian tensor category $K^{\circ a}_\Delta$ has all the colimits. Moreover, the counit of adjunction $M^a_*\to M$ is a natural isomorphism to the identity functor and there is an isomorphism $M^a_*\simeq M$. For $A\in K^{\circ a}_\Delta-\Alg$ and for $M,N\in A-\Mod$, there exists a tensor product defined by the formula
\begin{align*}
    M\otimes_A N:=(M_*\otimes_{A_*}N_*)^a
\end{align*}
which also gives $A-\Mod$ the structure of an abelian tensor category and hence can define module and algebra objects over it. We obtained an internal Hom functor as
\begin{align*}
    \alHom_A(M,N):=\Hom_{A-\Mod}(M,N)^a
\end{align*}
which is the right adjoint to the above tensor product functor. The localization functor $M\mapsto M^a$ also has a left adjoint given by
\begin{align*}
    (\bullet)_!:K^{\circ a}_\Delta-\Mod&\to K^\circ_\Delta-\Mod\\
                                      M&\mapsto \mathfrak{h}\otimes_{K^\circ_\Delta}M_*
\end{align*}
which is an exact functor, because $\mathfrak{h}$ is flat $K^{\circ}_\Delta-$module.

\begin{definition}[\cite{Sch},\S4]
    Let $A$ be $K^{\circ a}_\Delta-$algebra, $B$ be an $A-$algebra and $\mu:B\otimes_A B\to B$ denotes the diagonal morphism.
   \begin{enumerate}
    \item[(i)] An $A$-module $M$ is \emph{flat} if the functor $X \mapsto M \otimes_A X$ is exact. If $R$ is a $K^\circ_\Delta$-algebra and $N$ is an $R$-module, then the $R^a$-module $N^a$ is flat if and only if for all $R$-modules $X$ and all $i > 0$, the module $\Tor^R_i(N, X)$ is almost zero.
    \item[(ii)] An $A-$module $M$ is \emph{almost projective} if the functor $X\mapsto \alHom_A(M,X)$ on $A-$modules is exact. If $R$ is a $K^\circ_\Delta-$algebra and $N$ is an $R-$module, then $N^a$ is almost projective over $R^\circ$ if and only if for all $R-$modules $X$ and all $i>0$, the module $\Ext^i_R(N,X)$ is almost zero.
    \item[(iii)] If $R$ is a $K^\circ_\Delta-$algebra and $N$ is an $R-$module, then $M=N^a$ is said to be an \emph{almost finitely generated} (resp. \emph{almost finitely presented}) $R^a-$module if and only if for all $\epsilon\in\mathfrak{h}$, there is some finitely generated (resp. finitely presented) $R-$module $N_\epsilon$ with a map $f_\epsilon:N_\epsilon\to N$ such that  the kernel and cokernel of $f_\epsilon$ are annihilated by $\epsilon$. We say that $M$ is \emph{uniformly almost finitely generated} if there is some integer $n$ such that $N_\epsilon$ can be chosen to be generated by $n$ elements, for all $\epsilon$.
    \item[(iv)] An $A-$module $M$ is flat and almost finitely presented if and only if it is almost projective and almost finitely generated. In this case, $M$ is called \emph{finite projective}. If additionally $M$ is uniformly finitely generated, then $M$ is called \emph{uniformly finite projective}.
    \item[(v)] The morphism $A\to B$ is said to be \emph{unramified} if there is some element $e\in (B\otimes_A B)_*$ such that $e^2=e,\mu(e)=1$ and $xe=0$ for all $x\in\ker(\mu)_*$.
    \item[(vi)] The morphism $A\to B$ is said to be \emph{\'etale} if it is unramified and $B$ is a flat $A-$module.
    \item[(vii)] A morphism $A\to B$ of $K^{\circ a}_\Delta-$algebras is called \emph{finite \'etale} if it is \'etale and $B$ is an almost finitely presented $A-$module. The category of finite \'etale $A-$algebras is denoted by $A_{\fet}$. In this case, $B$ is a finite projective $A$-module.
   \end{enumerate}

\end{definition}

We will now give the analogous lemma to Lemma 5.3 \cite{Sch}.

\begin{lemma}[\cite{Sch},Lemma 5.3]\label{lemma5.1}
	Let $M$ be a $K^{\circ a}_\Delta-$module. Then
	\begin{enumerate}
		\item[(i).] The module $M$ is flat over $K^{\circ a}_\Delta$ if and only if $M_*$ is flat over $K^\circ_\Delta$ if and only if $M_*$ has no $\varpi-$torsion.
		\item[(ii).] If $N$ is a flat $K^\circ_\Delta-$module and $M=N^a$, then $M$ is flat over $K^{\circ a}$ and we have
		\begin{align*}
			M_*=\Big\{x\in N\left[\frac{1}{\varpi}\right] \mid \epsilon x\in N, \forall \epsilon\in\mathfrak{h}\Big\} .
		\end{align*}
		\item[(iii).] If $M$ is flat over $K^{\circ a}_\Delta$, then for all $x\in K^\circ_\Delta$, we have $(xM)_*=xM_*$. Moreover, $M_*/xM_*\subset (M/xM)_*$, and for all $\epsilon\in\mathfrak{h}$ the image of $(M/x\epsilon M)_*$ in $(M/xM)_*$ is equal to $M_*/xM_*$.
		\item[(iv).] If $M$ is flat over $K^{\circ a}_\Delta$, then $M$ is $\varpi-$adically complete if and only if $M_*$ is $\varpi-$adically complete.
	\end{enumerate}
\end{lemma}

\begin{proof}
The same proof is given in \cite{Sch}, but we recall it here in this context.
\begin{enumerate}
    \item[(i)] If $M_*$ is a flat $K^\circ_\Delta$-module, then $\Tor^{K^\circ_\Delta}_i(M_*,N)=0$ for every $i>0$ and all $K^\circ_\Delta$-modules $N$. Therefore, $M$ is a flat $K^{\circ a}_\Delta$-algebra.

    Conversely, assume $M$ is flat as a $K^{\circ a}_\Delta$-module. Then for every $K^\circ_\Delta$-module $N$ and every $i > 0$, $\Tor^{K^\circ_\Delta}_i(M_*, N)$ is almost zero. Taking $i=1$ and $N = K^\circ_\Delta/\varpi K^\circ_\Delta$, we get that $\Tor^{K^\circ}_1(K^\circ_\Delta/(\varpi),M_*)=M_*[\varpi]$, which is the kernel of multiplication by $\varpi$, is almost zero. However, $M_* = \Hom_{K^{\circ a}_\Delta}(K^{\circ a}_\Delta, M) = \Hom_{K^\circ_\Delta}(\mathfrak{h}, M_*)$ has no almost zero elements. Indeed, if $f\in\Hom_{K^\circ_\Delta}(\mathfrak{h},M_*)$ is almost zero, then for every $\epsilon\in\mathfrak{h}$, we have $\epsilon f=0$. Suppose $x\in\mathfrak{h}$; then since $\mathfrak{h}^2=\mathfrak{h}$, there exist $\epsilon_i,\epsilon'_i\in\mathfrak{h}$, $i=1,\dots,m$ such that $x=\sum\limits_{i=1}^m\epsilon_i\epsilon'_i$. Thus $f(x)=\sum\limits_{i=1}^m\epsilon_if(\epsilon'_i)=0$. Therefore, $f=0$.

    \item[(ii)] The first part follows immediately from (i). For the second part, we claim that there is a bijection
    \begin{align*}
        \theta: \Hom_{K^\circ_\Delta}(\mathfrak{h}, N) \leftrightarrow \left\{ x \in \Hom_{K_\Delta}\left(K_\Delta, N\left[\frac{1}{\varpi}\right]\right) \,\middle|\, \forall \epsilon \in \mathfrak{h}, \, \epsilon x \in N \right\}.
    \end{align*}
    Indeed, since $N$ is a flat $K^\circ_\Delta$-module and  $N\left[\frac{1}{\varpi}\right]=\Hom_{K^\circ_\Delta}\left(K^\circ_\Delta,N\right)\left[\frac{1}{\varpi}\right]=\Hom_{K_\Delta}\left(K_\Delta,N\left[\frac{1}{\varpi}\right]\right)$, there is an inclusion map and a canonical map into the localization. The image is determined by the image of $1$; then $\epsilon f(1)=f(\epsilon)\in N$, which proves the claim. The image identifies with $\Big\{x\in N\left[\frac{1}{\varpi}\right]\mid \epsilon x\in N, \forall \epsilon\in\mathfrak{h}\Big\}$.

    \item[(iii)] The proof goes exactly along the same lines as in \cite{Sch}, with $K^{\circ}$ replaced by $K^\circ_\Delta$ and $\mathfrak{m}$ replaced by $\mathfrak{h}$.

    \item[(iv)] Since the functors $M \mapsto M_*$ and $N \mapsto N^a$ are left adjoints, they commute with inverse limits. Hence, if $M$ is $\varpi$-adically complete, then
    \begin{align*}
        M_*=(\varprojlim M/\varpi^n)_*=\varprojlim(M/\varpi^n)_*=\varprojlim M_*/\varpi^n
    \end{align*}
    since from part (iii) we have $\varprojlim (M/\varpi^n M)_*=\varprojlim M_*/\varpi^{n-1}M_*=\varprojlim M_*/\varpi^n$. This shows that $M_*$ is $\varpi$-adically complete.

    Conversely, if $M_*$ is $\varpi$-adically complete, then
    \begin{align*}
        M = (M_*)^a = \left(\varprojlim M_* / \varpi^n\right)^a = \varprojlim (M_* / \varpi^n)^a = \varprojlim M /\varpi^n.
    \end{align*}
    Thus, $M$ is $\varpi$-adically complete.
  \end{enumerate}
\end{proof}


\section{Perfectoid Algebras over $K_\Delta$}

Throughout this section, we fix a $\Delta$-perfectoid field $K_\Delta$. We also fix a pseudo-uniformizing ideal $I_\varpi=(\varpi)$ generated by the pseudo-uniformizer $\varpi$.

\begin{definition}
   \begin{enumerate}
       \item[(i)] A \emph{perfectoid $K_\Delta$-algebra} is a Banach $K_\Delta$-algebra $R$ such that the subset $R^\circ \subset R$ of power-bounded elements is open and bounded, and the Frobenius morphism
		\[
		\Phi: R^\circ /\varpi \to R^\circ /\varpi
		\]
		is surjective. Morphisms between perfectoid $K_\Delta$-algebras are continuous morphisms of $K_\Delta$-algebras.
		
       \item[(ii)] A \emph{perfectoid $K^{\circ a}_\Delta$-algebra} is an $I_\varpi$-adically complete and flat algebra $A$ over $K^{\circ a}_\Delta$ such that the Frobenius morphism induces an isomorphism
        \begin{align*}
            \Phi: A / \varpi^{1/p} \xrightarrow{\sim} A /\varpi.
        \end{align*}
        Morphisms of perfectoid $K^{\circ a}_\Delta$-algebras are morphisms of $K^{\circ a}_\Delta$-algebras.
		
       \item[(iii)] A \emph{perfectoid $K^{\circ a}_\Delta/\varpi$-algebra} is a flat $K^{\circ a}_\Delta/\varpi$-algebra $\overline{A}$ such that the Frobenius morphism induces an isomorphism
        \begin{align*}
            \Phi: \overline{A} / \varpi^{1/p} \xrightarrow{\sim} \overline{A}.
        \end{align*}
        Morphisms of perfectoid $K^{\circ a}_\Delta/\varpi$-algebras are morphisms of $K^{\circ a}_\Delta/\varpi$-algebras.
   \end{enumerate}
\end{definition}

 We have the following categories:
 \begin{itemize}
     \item $K_\Delta-\Perf$ the category of perfectoid $K_\Delta-$algebras;
     \item $K^{\circ a}_\Delta-\Perf$ the category of perfectoid $K^{\circ a}_\Delta-$algebras;
     \item $K^{\circ a}_\Delta-\Perf$ the category of perfectoid $K^{\circ a}_\Delta-$algebras.
 \end{itemize}

   We prove the analogous result given in \cite{Sch} in this context, and the same proof will also work here.
  \begin{theorem}\label{Theorem 6.2}
	The category of perfectoid $K_\Delta-$algebras and perfectoid $K^\flat_\Delta-$algebras is equivalent. The equivalences proceed as follows
    \begin{align*}
        K_\Delta-\Perf\cong K^{\circ a}_\Delta-\Perf\cong K^{\circ a}_\Delta/\varpi-\Perf= K^{\flat\circ a}_\Delta/\varpi^\flat-\Perf\cong K^{\flat\circ a}_\Delta-\Perf\cong K^\flat_\Delta-\Perf.
    \end{align*}	
  \end{theorem}

Note that it is enough to prove the result for the first half of the chain of equivalences, and the same proof will work for the other half.

\begin{proposition}\label{Propostion 6.3}
    The functor 
    \begin{align*}
        K_\Delta-\Perf&\to K^{\circ a}_\Delta-\Perf\\
                     R&\mapsto R^{\circ a}
    \end{align*}
    is an equivalence of categories.
\end{proposition}

\begin{proof}
  The same proof given in Proposition 5.5, Lemma 5.6, and Lemma 5.7 of \cite{Sch} works in this setup. For the sake of completeness, let us recall the proof. We shall prove the results in steps:

\textbf{Step 1: $R^{\circ a}$ is a perfectoid $K^{\circ a}_\Delta$-algebra.} To show this, we will demonstrate that $\Phi$ induces an isomorphism $R^\circ/\varpi^{1/p}\cong R^\circ/\varpi$. Since $\Phi: R^\circ/\varpi\to R^\circ/\varpi$ is surjective and $\varpi^{1/p} R^\circ\subset \ker\Phi$, $\Phi$ factors through $R^\circ/\varpi^{1/p}$, and thus the induced map $R^\circ/\varpi^{1/p}\to R^\circ/\varpi$ is surjective. We will now prove injectivity. Suppose $x\in R^\circ$ is such that $x^p\in\varpi R^\circ$. Then $\frac{x^p}{\varpi}\in R^\circ$, which implies $\frac{x}{\varpi^{1/p}}\in R^\circ$, i.e., $x\in\varpi^{1/p}R^\circ$. Therefore, $\Phi$ induces a bijection $R^{\circ a}/\varpi^{1/p}\cong R^\circ/\varpi$. Moreover, $R^{\circ a}$ is flat because $R^\circ$, being torsion-free, is a flat $K^\circ_\Delta$-algebra, using Proposition \ref{Flatness criterion}. Therefore, $R^{\circ a}$ is a perfectoid $K^{\circ a}_\Delta$-algebra.

\textbf{Step 2: Construction of a quasi-inverse.} We claim that the functor
\begin{align*}
K^{\circ a}_\Delta-\Perf&\to K_\Delta-\Perf,\\
A&\mapsto A_*\left[\frac{1}{\varpi}\right]
\end{align*}
is the quasi-inverse to the functor $R\mapsto R^{\circ a}$. Set $R=A_*\left[\frac{1}{\varpi}\right]$. We begin by giving $R$ a Banach algebra structure. Define a family of norms on $R$ indexed by $s\in \mathfrak{S}$ as
\begin{align*}
\|x\|_s&:=\inf\{|\lambda|^{-1}_s\mid \beta_s(\lambda)\in K^\times_s\text{ and } \lambda x\otimes 1\in A_*\otimes_{K^\circ_\Delta}K^\circ_s\}.
\end{align*}

Let us verify that this is indeed a family of norms on $R$.
\begin{itemize}
\item Suppose $x \in R$ and $\lambda \in K_\Delta$. If $\beta_s(\lambda) = 0$, then $\lambda x \otimes 1 = x \otimes \beta_s(\lambda) = 0$; hence, for all $\lambda_1 \in K_\Delta$ with $\beta_s(\lambda_1) \neq 0$, we have $\lambda_1(\lambda x \otimes 1) = 0$. It follows that $\|\lambda x\|_s = 0 = |\lambda|_s \|x\|_s$. If $\beta_s(\lambda) \neq 0$, we have:
\begin{align*}
    \|\lambda x\|_s &= \inf \{ |\lambda_1|^{-1}_s \mid \beta_s(\lambda_1) \in K^\times_s \text{ and } \lambda_1(\lambda x) \otimes 1 \in A_* \otimes_{K^\circ_\Delta} K^\circ_s \} \\
    &= |\lambda|_s \inf \{ |\lambda \lambda_1|^{-1}_s \mid \beta_s(\lambda \lambda_1) \in K^\times_s \text{ and } (\lambda \lambda_1) x \otimes 1 \in A_* \otimes_{K^\circ_\Delta} K^\circ_s \} \\
    &= |\lambda|_s \|x\|_s.
\end{align*}

\item Suppose that for some $x \in R$, we have $\|x\|_s = 0$ for every $s \in \mathfrak{S}$. Then for every $\lambda \in K^\times_\Delta$, we have $\lambda x \otimes 1 \in A_* \otimes_{K^\circ_\Delta} K^\circ_s$. Consider the following map:
\begin{align*}
    A_* \cong A_* \otimes_{K^\circ_\Delta} K^\circ_\Delta &\to \prod_{s \in \mathfrak{S}} A_* \otimes_{K^\circ_\Delta} K^\circ_s \subset \prod_{s \in \mathfrak{S}} R \otimes_{K_\Delta} K_s \cong \prod_{s \in \mathfrak{S}} R_{\mathfrak{P}_s} \\
    \sum_{i=1}^n x_i \otimes \lambda_i &\mapsto \left( \sum_{i=1}^n x_i \otimes \beta_s(\lambda_i) \right)_{s \in \mathfrak{S}}
\end{align*}
Denote the image of $x \in A_*$ under the above map by $(x_s)_{s \in \mathfrak{S}} \in \prod_{s \in \mathfrak{S}} R_{\mathfrak{P}_s}$. Note that if $(x_s)_{s \in \mathfrak{S}} = 0$, then $x = 0$; therefore, the map is injective. This implies $\lambda x \in A_*$ for every $\lambda \in K^\times_\Delta$. In particular, for every $n \geq 0$, we have $\varpi^{-n} x \in A_*$, which shows that $x = 0$.

\item Clearly, $\|xy\|_s \leq \|x\|_s \|y\|_s$.
\end{itemize}

Consequently, $\|\cdot\|_s$ defines a family of norms. The $K_\Delta$-algebra $R$ is complete with respect to the norm $\|x\| = \sup_{s \in \mathfrak{S}} \|x\|_s$, as it is $\varpi$-adically complete. Thus, $R$ is a Banach $K_\Delta$-algebra.

 Next, we want to show that $R$ is a perfectoid $K_\Delta$-algebra, i.e., the Frobenius $\Phi: R^\circ/\varpi\to R^\circ/\varpi$ is surjective. We claim that $R^\circ=A_*$. To see this, we will first show that
 \begin{align*}
   A_*=\{x\in R\mid \|x\|\leq 1\}.
 \end{align*}
 Indeed, if $x\in A_*$, then $1\cdot x=x\in A_*$, and hence $\|x\|_s\leq 1$ for every $s\in\mathfrak{S}$. Conversely, suppose $x\notin A_*$. Then there exists $m\in\mathbb{N}^*$ such that $\varpi^{m} x\notin A_*$ but $\varpi^{m+1}x\in A_*$. Suppose $\lambda\in K_\Delta\setminus \mathfrak{P}_s$ such that $\beta_s(\lambda)\vartheta_s(x)\in A_*\otimes_{K^\circ_\Delta}K^\circ_s$. Then $\beta_s(\lambda)\mid \beta_s(\varpi^{m})$, i.e., $|\lambda|_s\leq |\varpi|_s=|\varpi|$. Therefore, $\|x\|_s\geq |\lambda|^{-1}_s\geq |\varpi|^{-1}$ for every $s\in\mathfrak{S}$. Thus $\|x\|=\sup\limits_{s\in\mathfrak{S}}\{\|x\|_s\}\geq |\varpi|^{-1}>1$.

 We will now show that $A_*$ is the set of power-bounded elements of $R^\circ$. If $x$ is a power bounded elements of $R$, then $\epsilon x$ is topologically nilpotent and hence for some $N\in\mathbb{N}$, we have $(\epsilon x)^{p^N}\in A_*$, since $A_*$ is open.\\ 

 \textbf{Claim.}\emph{ Assume that $x\in R$ satisfies $x^p\in A_*$. Then $x\in A_*$.}\\
 
 \emph{Proof of the Claim.} Since $R=A_*\left[\frac{1}{\varpi}\right]$, there exists a positive integer $k$ such that $y=\varpi^{k/p}x\in A_*$. This implies $y^p=\varpi^kx^p\in\varpi A_*$. As $A_*$ is torsion-free and $\varpi^{1/p}$ is not a zero divisor, the condition $y^p\in A_*$ implies $y\in \varpi^{1/p}A_*$. Substituting $y=\varpi^{k/p}x$, we get $\varpi^{k/p}x\in\varpi^{1/p}A_*$, which further implies $\varpi^{\frac{k-1}{p}}x\in A_*$. Repeating this argument $k$ times, we get $x\in A_*$. This completes the proof of the claim.\\
 
 We need to show that the Frobenius $\Phi: A_*/\varpi\to A_*/\varpi$ is surjective. It is almost surjective by assumption; thus, it is enough to show surjectivity after composition with the canonical surjection $A_*/\varpi\to A_*/\varpi\to A_*/\mathfrak{h}$. Let $x\in A_*$. Then by almost surjectivity, there exists $y\in A_*$ such that $\varpi^{1/p}x\equiv y^p\mod\varpi$. Consider the element $z=\frac{y}{\varpi^{1/p^2}}\in R$. Then we get $z^p=\frac{y^p}{\varpi^{1/p}}\equiv x\mod \varpi^{1-1/p}$. This proves surjectivity $A_*/\varpi\to A_*/\mathfrak{h}$.

\end{proof}

\begin{proposition}[\cite{Sch},Proposition 5.8]
    Let $R$ be a perfectoid $K_\Delta-$algebra. Then $R$ is reduced.
\end{proposition}

\begin{proof}
     Suppose $x$ is nilpotent, then there exists $n>1$ such that $x^n=0$. Then $xK_\Delta\subset R^\circ$ and hence $x/\varpi^n\in R^\circ$ for all $n\geq 1$ giving $x\in\bigcap\limits_{n\in\mathbb{N}}\varpi^n R^\circ=\{0\}$ as $\varpi$ is a topologically nilpotent unit and $R^\circ$ is bounded and $R$ is $\varpi-$adically separated. 
\end{proof}

\begin{proposition}
    Let $K_\Delta$ be of characteristic $p$, and let $R$ be a Banach $K_\Delta-$algebra such that the set of power-bounded elements $R^\circ\subset R$ is open and bounded. Then $R$ is perfectoid if and only if $R$ is perfect.
\end{proposition}

\begin{proof}
    Suppose $R$ is perfectoid. $R$ is reduced and $R^\circ/\varpi=R^\circ$ as $|p|\leq |\varpi|<1$. Thus, the ring is perfect. Conversely, suppose $R$ is a perfect ring of characteristic $p$, then if $x\in R^\circ$, then $x^p\in R^\circ$. Thus, the Frobenius $R^\circ/\varpi\to R^\circ/\varpi$ is surjective. 
\end{proof}

\begin{proposition}[\cite{Sch} Proposition-5.13]
    \begin{enumerate}
        \item[(i).] Let $R$ be a perfect $\mathbb{F}_p-$algebra. Then $\mathbb{L}_{R/\mathbb{F}_p}\cong 0$.
        \item[(ii).] Let $R\to S$ be a morphism of $\mathbb{F}_p-$algebras. Let $R_{(\Phi)}$ is the ring $R$, viewed as an $R$-algebra through the Frobenius, and define $S_{(\Phi)}$ similarly. Assume that the relative Frobenius $\Phi_{S/R}$ induces an isomorphism
        \begin{align*}
            R_{(\Phi)}\otimes^{\mathbb{L}}_RS\to S_{(\Phi)}
        \end{align*}
        in $D(R)$. Then $\mathbb{L}_{S/R}\cong 0$.
    \end{enumerate}
\end{proposition}

\begin{corollary}[\cite{Sch} Corollary 5.16]\label{Cor6.1}
Let $\overline{A}$ be a perfectoid $K^{\circ a}_\Delta/\varpi-$algebra. Then $\mathbb{L}^a_{\overline{A}/(K^{\circ a}_\Delta/\varpi)}\cong 0$.
\end{corollary}

We will now prove the second isomorphism.

\begin{proposition}
    The functor 
    \begin{align*}
        K^{\circ a}_\Delta-\Perf&\to (K^{\circ a}/\varpi)-\Perf\\
                              A &\mapsto \overline{A}=A/\varpi
    \end{align*}
    is an equivalence of categories.
\end{proposition}

\begin{proof}
   The same proof as \cite{Sch}, Theorem 5.10, works in this case. Let $\overline{A}$ be a perfectoid $K^{\circ a}_\Delta/\varpi$-algebra. Following \cite{Sch}, we can lift $\overline{A}$ to a flat $(K^\circ_\Delta/\varpi^n)^a$-algebra $\overline{A}_n$ if the almost cotangent complex $\mathbb{L}^a_{\overline{A}_n/(K^\circ_\Delta/\varpi)}$ vanishes.

We will establish the existence of such lifts $\overline{A}_n$ of $\overline{A}$ as flat $(K^{\circ}/\varpi^n)^a$-algebras using induction on $n$. The base case $n=1$ is immediate. Assume that there exists a flat $(K^\circ_\Delta/\varpi^{n-1})^a$-algebra $\overline{A}_{n-1}$ lifting $\overline{A}$ such that $\mathbb{L}^a_{\overline{A}_{n-1}/(K^\circ_\Delta/\varpi^{n-1})^a} = 0$. This vanishing implies the existence of a flat $(K^\circ/\varpi^n)^a$-algebra $\overline{A}_n$ lifting $\overline{A}_{n-1}$, i.e., there exists a short exact sequence:
\begin{align*}
0 &\longrightarrow \overline{A} \longrightarrow \overline{A}_n \longrightarrow \overline{A}_{n-1} \longrightarrow 0.
\end{align*}
Tensoring this sequence with the almost cotangent complex $\mathbb{L}^a_{\bullet/(K^\circ_\Delta/\varpi^n)^a}$, we obtain a distinguished triangle:
\begin{align*}
\mathbb{L}^a_{\overline{A}/(K^\circ_\Delta/\varpi)^a} \longrightarrow \mathbb{L}^a_{\overline{A}_n/(K^\circ_\Delta/\varpi^n)^a} \longrightarrow \mathbb{L}^a_{\overline{A}_{n-1}/(K^\circ_\Delta/\varpi^{n-1})^a} \longrightarrow
\end{align*}
From this distinguished triangle and the inductive hypothesis $\mathbb{L}^a_{\overline{A}_{n-1}/(K^\circ_\Delta/\varpi^{n-1})^a} = 0$, we deduce that $\mathbb{L}^a_{\overline{A}_n/(K^\circ_\Delta/\varpi^n)^a} = 0$. This completes the inductive step and proves the existence of the lifts $\overline{A}_n$. Then $A=\varprojlim\limits_{n} \overline{A}_n$ is a perfectoid $K^{\circ a}_\Delta$-algebra.
\end{proof}

\begin{corollary}
    The functor 
    \begin{align*}
        K_\Delta-\Perf&\to K^\flat_\Delta-\Perf\\
                    R &\mapsto R^\flat:=\varprojlim\limits_{\Phi} R
    \end{align*}
    is an equivalence of categories. This equivalence is called the \emph{tilting equivalence}. Moreover, we have a multiplicative map $\#: R^\flat\to R$ defined by $x^\#=x^{(0)}$, where $R^\flat=\Big\{(x^{(n)}\in \prod\limits_{n\in\mathbb{N}}R\mid (x^{(n+1)})^p=x^{(n)})\Big\}$.
\end{corollary}

\begin{proof}
    To prove the equivalence, we only need to identify the middle part of the chain of isomorphism appearing in the theorem \ref{Theorem 6.2}. But note that the composition of $\#:R^\flat\to R$ and $R\to R/\varpi$ is a ring homomorphism whose kernel is $\varpi^\flat R^{\flat\circ}$. Taking the image under almost functor gives the equality of the middle part.
\end{proof}

\begin{proposition}
    Let $R$ be a perfectoid $K_\Delta-$algebra with tilt $R^\flat$. Then $R$ is a VNR ring if and only if $R^\flat$ is a VNR ring. Moreover, $R$ is a field if and only if $R^\flat$ is a field.
\end{proposition}

\begin{proof}
Suppose $R$ is a perfectoid VNR ring, and consider $x=(x^{(n)})_{n}\in R^\flat$. We need to show that either $x$ is a zero divisor or a unit. We have $R^\flat=\varprojlim\limits_{x\mapsto x^p} R$. Suppose $x\in R^\flat$ is not a unit; then $x^\#$ is a zero divisor in $R$. Thus, there exists $y\in R^\flat$ such that $x^\#y^\#=(xy)^\#=0$. Lifting this expression shows that $R^\flat$ is a VNR ring.

Conversely, suppose $R^\flat$ is a VNR ring. We claim that the family of spectral norms
\begin{align*}
\|x\|_{s,R}=\inf\{|\lambda|^{-1}_s\mid \beta_s(\lambda)\in K^\times_s, \lambda x\otimes 1\in R^\circ\otimes_{K^\circ_\Delta}K^\circ_s\}, s\in\mathfrak{S}
\end{align*}
are multiplicative. Consider $x,y\in R$. After multiplying by suitable elements of $K_\Delta\setminus\mathfrak{P}_s$, we can assume that $x,y\in R^\circ\setminus\varpi^{1/p}$. Suppose $x^\flat,y^\flat\in R^{\flat\circ}$ are such that $x-(x^\flat)^\#,y-(y^\flat)^\#\in \varpi R^\circ$. Then we have
\begin{align*}
\|x-(x^\flat)^\#\|_{s,R}\leq |\varpi|\implies \|x^{\flat}\|_{s,R^\flat}=\|x\|_{s,R}.
\end{align*}
Similarly, we have $\|y\|_{s,R}=\|y^\flat\|_{s,R^\flat}$. Also, we have
\begin{align*}
\|xy-(x^\flat)^\#(y^\flat)^\#\|_{s,R}&=\|xy-(x^\flat)^\#y+(x^\flat)^\#y- (x^\flat)^\#(y^\flat)^\#\|_{s,R}\\
&\leq\max\{\|x-(x^\flat)^\#\|_{s,R}\|y\|_{s,R},\|(x^\flat)^\#\|_{s,R}\|y-(y^\flat)^\#\|\}\leq |\varpi|
\end{align*}
This is only possible if $\|xy\|_{s,R}=\|x^\flat y^\flat\|_{s,R^\flat}$. Therefore, for every $s\in\mathfrak{S}$, the spectral norm $\|\bullet\|_{s,R}$ is multiplicative.

Suppose $x\in R$ is such that $x^\flat$ is defined as above. Suppose it is not a zero divisor; then for every $s\in\mathfrak{S}$, we have
\begin{align*}
\left\|1-\frac{x}{(x^\flat)^\#}\right\|_{s,R}=\left\|\frac{(x^\flat)^\#-x}{(x^\flat)^\#}\right\|_{s,R}=\frac{\|(x^\flat)^\#-x\|_{s,R}}{\|(x^\flat)^\#\|_{s,R}}\leq \frac{|\varpi|}{|\varpi|^{1/p}}=|\varpi|^{\frac{p-1}{p}}.
\end{align*}
The topology in $R$ is given by $\|x\|_R=\sup\limits_{s\in\mathfrak{S}}\|x\|_{s,R}$. Thus, $\left\|1-\frac{x}{(x^\flat)^\#}\right\|_R<1$, and hence $\sum\limits_{i=0}^\infty (-1)^i\left(1-\frac{x}{(x^\flat)^\#}\right)^i$ converges to $\frac{(x^\flat)^\#}{x}$. But $R^\flat$ is a perfectoid VNR ring, and $x^\flat$ is a unit in $R^\flat$ as it is not a zero divisor, so we obtain that $x$ is a unit in $R$. If $x$ is a zero divisor, then $x^\flat$ is also a zero divisor. For the second part, note that $R$ is a field if and only if for all $x\in R$, we have $\|x\|_{s,R}$ is zero for all but one $s\in\mathfrak{S}$. Indeed, if $\|\bullet\|_{s,R}$ is zero for all but one $s\in\mathfrak{S}$, say for $t\in\mathfrak{S}$, then $\Supp\|\bullet\|_t=\ker\|\bullet\|_t$ is a prime ideal. But $\bigcap\limits_{s\in\mathfrak{S}}\ker\|\bullet\|_s=\{0\}$, and hence $\ker\|\bullet\|_t=\{0\}$. Therefore, there exist no non-zero zero divisors in $R$, proving every element is a unit. From the first part, we obtain the result.
\end{proof}

\begin{remark}
    This proposition is incomplete in the sense that we do not know if every characteristic zero untilts of a perfectoid VNR ring can be written as a completed tensor product of perfectoid fields.
\end{remark}

\begin{remark}
    The analogue of Proposition 5.17 and Proposition 5.20 of \cite{Sch} holds in this case with the same proof. In particular, if we set $R=K_\Delta\left<T^{1/p^\infty}_1,\cdots,T^{1/p^\infty}_n\right>$. Then $R$ is perfectoid $K_\Delta$-algebra with tilt $K^\flat_\Delta\left<T^{1/p^\infty}_1,\cdots,T^{1/p^\infty}_n\right>$.
\end{remark}

\begin{lemma}
    Let $\overline{A}$ be a perfectoid $K^{\circ a}_\Delta/\varpi-$algebra, and let $\overline{B}$ be a finite etale $\overline{A}-$algebra. Then $\overline{B}$ is a perfectoid $K^{\circ a}_\Delta/\varpi-$algebra.
\end{lemma}

\begin{proof}
    The proof follows the same line of reasoning as Proposition 5.22 of \cite{Sch}.
\end{proof}

\begin{proposition}[Almost purity over $K_\Delta$ of characteristics $p$]
    Let $K_\Delta$ be a perfectoid $\Delta$-field of characteristics $\say{p}$ and let $R$ be a perfectoid $K_\Delta-$algebra, and let $S/R$ be finite \'{e}tale. Then $S$ is perfectoid and $S^{\circ a}$ is finite \'{e}tale over $R^{\circ a}$. Moreover, $S^{\circ a}$ is a uniformly finite projective $R^{\circ a}-$module.
\end{proposition}

\begin{proof}
    The same proof of Proposition 5.23 of \cite{Sch}, but we will recall the proof for completeness. We will prove this in steps:\\
    \textbf{Step-1: $S$ is perfectoid over $K_\Delta$:} Since $R$ is perfectoid of characteristics $p$, it is a perfect $\mathbb{F}_p-$algebra, meaning the Frobenius endomorphism $\Phi:R\to R$ is bijective. From Theorem 3.5.13 of \cite{ART} applied to the classical case, the diagram
    \[
    \begin{tikzcd}
        R\arrow[r]\arrow[d,"\Phi"] & S\arrow[d,"\Phi"]\\
        R\arrow[r]           & S
    \end{tikzcd}
    \]
    is cocartesian, showing $\Phi:S\to S$ is a bijective. We can define topology on $S$ if we can define a topology for a free $ R-$module $S$ because $S$ is a projective $R-$module, i.e., a direct summand of a free module. For doing so, we will define an analogue of Definition 5.4.10 \cite{ART}. Let $I$ be the ideal of definition of $R^\circ$, and $S_\circ$ is a finite $R^\circ-$module such that $S_\circ\left[\frac{1}{\varpi}\right]=S$. Then by declaring the family of submodules $\{\varpi^{n}IS_\circ\}_{n\in\mathbb{N}}$ as a fundamental system of neighborhoods of $0$ which gives a topology on $S_\circ$ for which it is $I-$adically complete. The topology on $S$ is the finest topology such that the canonical inclusion $S_\circ\to S$ is an open map. Therefore $\{\varpi^n IS_0\}_{(n)\in\mathbb{N}}$ is a fundamental system of neighborhood of $0$ in $S$. We have to show that this topology is independent of the choice of $S_\circ$. Suppose $S'_\circ$ is another finite $R^\circ-$module with $S'_\circ\left[\frac{1}{\varpi}\right]=S$. Then for some $m>>0$, we have $\varpi^mS_\circ\subset S'_\circ$ and for some $n>>0$ with $\varpi^n S'_\circ\subset S_\circ$, proving the independence. In order to show $S$ is perfectoid, we need to show that the ring of power-bounded elements $S^\circ$ is open and bounded. For this, we will fix a finitely generated $R^\circ-$algebra $S_\circ$ as above and define
    \begin{align*}
        S^\perp_\circ:=\{x\in S\mid t_{S/R}(x,S_0)\subset R^\circ\}
    \end{align*}
    where $t_{S/R}:S\otimes_R S\to R$ is the perfect trace pairing. Then $S_\circ$ and $S^\perp_\circ$ are open and bounded. Suppose $Y$ is the integral closure of $R^\circ$ in $S$. As $S_\circ$ is a finite $R^\circ-$module, the elements of $S_\circ$ are integral and also trace maps integral elements to integral elements, we have $S_\circ\subset Y\subset S^\perp_\circ$. Therefore, $Y$ is open and bounded. But $S^{\circ a}=Y^a$, it follows that $S^\circ$ is open and bounded, hence $S$ is a perfectoid $K_\Delta$-algebra.\\
    \textbf{Step-2: $S^{\circ a}$ is a uniformly finite projective $R^{\circ a}-$module:} Since $S/R$ is unramified as it is \'etale, there exists an idempotent $e\in S\otimes_R S$. As $S=S^\circ\left[\frac{1}{\varpi}\right]$, there exists some $N$ large enough so that $\varpi^N e\in S^\circ\otimes_{R^\circ}S^\circ$ for all $\alpha\in\Delta$. Thus we can write $\varpi^N e=\sum\limits_{i=1}^n x_i\otimes y_i$, $\alpha\in\Delta$. Since Frobenius is bijective, we get $\varpi^{N/p^m} e=\sum\limits_{i=1}^n x^{1/p^m}_i\otimes y^{1/p^m}_i$ for all $m$. Thus for every $\epsilon\in\mathfrak{h}$, we can write $\epsilon e=\sum\limits_{i=1}^n a_i\otimes b_i$ for some $a_i,b_i\in S^\circ$ depending on $\epsilon$. Therefore, we get maps
    \begin{align*}
        S^\circ&\to R^{\circ n}\\
        s&\mapsto (t_{S/R}(s,b_1),\cdots,t_{S/R}(s,b_n))
    \end{align*}
    and
    \begin{align*}
        R^{\circ n}&\to S\\
        (r_1,\cdots,r_n)&\mapsto \sum\limits_{i=1}^n a_ir_i
    \end{align*}
    Then their composition is the multiplication by $\epsilon$ proving the fact that $S^{\circ a}$ is uniformly finite projective $R^{\circ a}-$module.\\
    \textbf{Step-3: $S^{\circ a}$ is an unramified $R^{\circ a}-$algebra:} By previous argument, $e$ defines an almost element with $e^2=e, \mu(e)=1$ and $xe=0$ for every $x\in(\ker\mu)_*$.
\end{proof}

\begin{theorem}[\cite{Sch},Theorem 4.17 ]
	Let $A$ be a $K^{\circ a}_\Delta$-algebra. Suppose that $A$ is flat over $K^{\circ a}_\Delta$ and $\varpi$-adically complete. Then the functor
	\[
	B \mapsto B \otimes_A A/\varpi
	\]
	induces an equivalence of categories
	\[
	A_{f\acute{e}t} \cong (A/\varpi)_{f\acute{e}t}.
	\]
	Any $B\in A_{\fet}$ is again flat over $K^{\circ a}_\Delta$ and $\varpi-$adically complete. Moreover, $B$ is a uniformly finite projective $A-$module if and only if $B \otimes_A A/\varpi$ is a uniformly finite projective $A/\varpi$-module.
\end{theorem}

\begin{proof}
  Since $I_\varpi=(\varpi)$ is a tight ideal in the sense of Definition~5.1.5 in \cite{ART}, Theorem~5.3.27 of \cite{ART} yields the equivalence of categories as stated in \cite{Sch}. 
\end{proof}  

Similar to \cite{Sch}, we have the following diagram

\[
\begin{tikzcd}
   R_{\fet} &\arrow[l] A_{\fet}\arrow[d]\arrow[r,"\cong"] & \overline{A}_{\fet}\arrow[d] & \arrow[l,swap,"\cong"] A^\flat_{\fet}\arrow[d] \arrow[r,"\cong"] & R^\flat_{\fet}\arrow[d]\\
   K_\Delta-\Perf  &\arrow[l,swap,"\cong"]K^{\circ a}_\Delta-\Perf\arrow[r,"\cong"] & (K^{\circ a}_\Delta/\varpi)-\Perf &\arrow[l,swap,"\cong"] K^{\flat\circ a}_\Delta-\Perf &\arrow[l,swap,"\cong"] K^\flat_\Delta-\Perf 
\end{tikzcd} 
\]


\section{Perfectoid Spaces over $K_\Delta$}

Through this section, we fix a perfectoid $\Delta$-field $K_\Delta$, a pseudo-uniformizing ideal $I_\varpi$ generated by $\varpi$. Furthermore, assume $|\varpi_\alpha|$ is all the same for every $\alpha\in\Delta$. We denote the corresponding tilt of $K_\Delta$ by $K^\flat_\Delta$, and the corresponding ideal $I^\flat_\varpi$ generated by  $\varpi^\flat$, respectively.

\begin{definition}
    An affinoid $K_\Delta-$algebra $(R,R^+)$ is called perfectoid if $R$ is a perfectoid $K_\Delta-$algebra.
\end{definition}

\begin{lemma}[Tilting Equivalence]
   The functor
   \begin{align*}
       \text{(Affinoid Perfectoid $K_\Delta$-algebras)}&\to \text{(affinoid perfectoid $K^\flat_\Delta$-algebras)}\\
       (R,R^+)&\mapsto (R^\flat,R^{\flat+})
    \end{align*}
   is an equivalence of categories. This correspondence induces isomorphism $R^{\flat+}/\varpi^\flat\cong R^+/\varpi$ and $R^{\flat+}=\varprojlim\limits_{x\mapsto x^p}R^+$.
\end{lemma}

\begin{proof}
    The proof follows the same line as in Lemma 6.2 of \cite{Sch}. Note that $\mathfrak{h}R^\circ \subset R^{\circ\circ}_\Delta \subset R^+ \subset R^\circ$, because topologically nilpotent elements are integral. Also, $\mathfrak{h}$ is the Jacobson radical of $K^\circ_\Delta$, and hence we have $\mathfrak{h}^\flat$ as the Jacobson radical of $K^{\circ\flat}_\Delta$. Using the above observation, we see that open integrally closed subrings of $R^\circ$ are in one-to-one correspondence with integrally closed subrings of $R^\circ/\mathfrak{h}$. Since $\varpi \in \mathfrak{h}$ and $\varpi^\flat \in \mathfrak{h}^\flat$, and $R^\circ/\varpi \cong R^{\circ\flat}/\varpi^\flat$, we deduce that integrally closed subrings of $R^\circ/\mathfrak{h}$ are in one-to-one correspondence with the integrally closed subrings of $K^{\flat\circ}/\mathfrak{h}^\flat$. Thus, this correspondence lifts to the other side, and we obtain a bijection between $R^+$ and $R^{\flat+}$.
\end{proof}

\begin{remark}\label{rationalsubsetmodification}
    If $U=U\left(\frac{f_1,\cdots.f_n}{g}\right)$, then to $f_1,\cdots f_n$ one can add $f_{n+1}=\varpi^N$ for a large $N$ without changing the rational subsets. Indeed, there exists $h_1,\cdots,h_n\in A$ such that $\sum\limits_{i=1}^n h_if_i=1$. Multiplying by $\varpi^N$ for some large $N$, we can assume $\varpi^N f_i\in R^+$ as $R^+\subset R$ is open. It is clear that $U\left(\frac{f_1,\cdots,f_{n+1}}{g}\right)\subseteq U\left(\frac{f_1,\cdots,f_n}{g}\right)=U$. Suppose $x\in U$, then 
    \begin{align*}
        |\varpi^N(x)|=\left|\sum\limits_{i=1}^n\varpi^N h_if_i\right|\leq \max|(\varpi^N h_i)(x)||f_i(x)|\leq |g(x)|.
    \end{align*}
    Thus $x\in U\left(\frac{f_1,\cdots,f_{n+1}}{g}\right)$ proving equality.
\end{remark}

Suppose $R$ is a perfectoid $K_\Delta$ algebra and $R^\flat$ is the tilt of $R$. We have a natural continuous map $\phi: \mathrm{Spa}(R, R^+) \to \mathrm{Spa}(R^\flat, R^{\flat+})$ defined as $\phi(x) = x^\flat$ with $|\bullet(x^\flat)| = |\bullet^\#(x)|$. Using Remark \ref{rationalsubsetmodification}, we see that this is a continuous morphism. Indeed, if $U = U\left(\frac{f_1, \cdots, f_n}{g}\right)$ is a rational subset of $\mathrm{Spa}(R^\flat, R^{\flat+})$ with $f_n = \varpi^{\flat N}$, then the inverse image of $U$ is given by $U^\# = U\left(\frac{f^\#_1, \cdots, f^\#_n}{g^\#}\right)$, which is rational as $f^\#_n = \varpi^N$ so that $f^\#_1R + \cdots + f^\#_nR = R$. We want to show that this is a homeomorphism identifying rational subsets. Along the same lines, we will prove the following lemmas as in \cite{Sch}.

\begin{lemma}[Lemma 6.4 \cite{Sch}]\label{Almost Rational Subset}
   Let $U=U\left(\frac{f_1,\cdots,f_n}{g}\right)\subset\spa(R^\flat,R^{\flat+})$ be a rational subset, with preimage $U^{\#}\subset\spa(R,R^+)$. Assume that all $f_i,g\in R^{\flat\circ}$ and that $f_n=\varpi^{\flat N/t}$ for some $N\in t\mathbb{N}$.
\begin{enumerate}
    \item[(i).] Consider the $\varpi-$adic completion
    \begin{align*}
        R^\circ\left<\left(\frac{f^\#_1}{g^\#}\right)^{1/p^\infty},\cdots,\left(\frac{f^\#_n}{g^\#}\right)^{1/p^\infty}\right>
    \end{align*}
    of the subring
    \begin{align*}
        R^\circ\left[\left(\frac{f^\#_1}{g^\#}\right)^{1/p^\infty},\cdots,\left(\frac{f^\#_n}{g^\#}\right)^{1/p^\infty}\right]\subset R\left[\frac{1}{g^\#}\right].
    \end{align*}
    Then $R^\circ\left<\left(\frac{f^\#_1}{g^\#}\right)^{1/p^\infty},\cdots,\left(\frac{f^\#_n}{g^\#}\right)^{1/p^\infty}\right>^a$ is a perfectoid $K^{\circ a}_\Delta-$algebra.
    \item[(ii).] The algebra $\mathcal{O}_X(U^\#)$ is a perfectoid $K_\Delta-$algebra, with associated perfectoid $K^{\circ a}_\Delta-$algebra
    \begin{align*}
        \mathcal{O}_X(U^\#)^{\circ a}&=R^\circ\left<\left(\frac{f^\#_1}{g^\#}\right)^{1/p^\infty},\cdots,\left(\frac{f^\#_n}{g^\#}\right)^{1/p^\infty}\right>^a.
    \end{align*}
    \item[(iii).] The tilt of $\mathcal{O}_X(U^\#)$ is given by $\mathcal{O}_{X^\flat}(U)$.
\end{enumerate}
\end{lemma}

\begin{proof}
    The same proof of Lemma 6.4 in \cite{Sch} can be adapted to this setting.
\end{proof}

We will now state the approximation lemma given in \cite{Sch}, Lemma 6.5, whose proof is just a minor modification.

\begin{lemma}[Approximation Lemma]\label{Approximation Lemma}
    Let $R=K_\Delta\left<T^{1/p^\infty}_1,\cdots,T^{1/p^\infty}_n\right>$. Let $f\in R^\circ$ be a homogeneous element of degree $d\in\mathbb{Z}\left[\frac{1}{p}\right]$. Then for any rational number $c\geq 0$ and any $\epsilon>0$, there exists an element
   \begin{align*}
       g_{c,\epsilon}\in R^{\flat\circ}=K^{\flat\circ}_\Delta\left<T^{1/p^\infty}_1,\cdots,T^{1/p^\infty}_n\right>
   \end{align*}
   homogeneous of degree $d$ such that for all $x\in X=\spa(R,R^\circ)$, we have 
   \begin{align*}
    |f(x)-g^\#_{c,\epsilon}(x)|\leq |\varpi|^{1-\epsilon}\max(|f(x)|,|\varpi_\alpha|^c).
   \end{align*}
\end{lemma}

\begin{proof}
   The same proof as in Lemma 6.5 of \cite{Sch} works in this case.
\end{proof}

Using the approximation lemma, we get the following corollary as in Corollary 6.7 \cite{Sch}.

\begin{corollary}
    Let $(R,R^+)$ be a perfectoid affinoid $K_\Delta-$algebra, with tilt $(R^\flat,R^{\flat+})$ and let $X=\spa(R,R^+), X^\flat=\spa(R^\flat,R^{\flat+})$.
    \begin{enumerate}
        \item[(i).] For any $f\in R$ and any $c\geq 0,\epsilon>0$, there exists $g_{c,\epsilon}\in R^{\flat}$ such that for all $x\in X$, we have
        \begin{align*}
            |f(x)-g^\#_{c,\epsilon}(x)|\leq |\varpi|^{1-\epsilon}\max(|f(x)|,|\varpi_\alpha|^c).
        \end{align*}
        \item[(ii).] For any $x\in X$, the completed residue field $\widehat{k(x)}$ is a perfectoid field.
        \item[(iii).] The morphism $X\to X^\flat$ induces a homeomorphism, identifying rational subsets.
    \end{enumerate}
\end{corollary}

\begin{proof}
    (i). Similarly as in Corollary 6.7 of \cite{Sch}, we can take $R^+=R^\circ$. We can write 
    \begin{align*}
        f=g^\#_0+b_1g^\#_1+\cdots+b_cg^\#_c+b_{c+1}f_{c+1}
    \end{align*}
    where $\varpi$ divides $b_i$, $(1\leq i\leq c+1)$ and $g_0,\cdots,g_c\in R^{\flat\circ}$ and $f_{c+1}\in R^\circ$. Note that we can assume $f=g^\#_0+\cdots+b_cg^\#_c$ by estimating the valuations. Consider the map
    \begin{align*}
        K_\Delta\left<T^{1/p^\infty}_0,\cdots,T^{1/p^\infty}_c\right>&\to R .\\
                                                        T_i^{1/p^m}  &\mapsto (g^{1/p^m}_i)^\#
    \end{align*}
    Then $f$ is the image of $T_0+b_1T_1+\cdots+b_cT^c$ on which applying the Approximation lemma gives the required element.

    (ii). ($K_\Delta$ has characteristic $p$) We know $\mathcal{O}_X(U)^{\circ a}$ is perfectoid for any rational subset $U$. Thus the $\varpi-$adic completion of $\mathcal{O}^{\circ a}_{X,x}$ is perfectoid $K^{\circ a}_\Delta-$algebra, hence $\widehat{k(x)}$ is a perfectoid $K_\Delta$-algebra. The structure map $K_\Delta\to \widehat{k(x)}$ has kernel as some prime ideal $\mathfrak{P}_s$ for some $s\in\mathfrak{S}$. Then $\widehat{k(x)}$ is a perfectoid field over $K_s$.

    (iii). We have already seen that $\phi:\spa(R,R^+)\to \spa(R^\flat,R^{\flat +})$ is continuous by showing the inverse image of a rational subset is rational. Now suppose $U\left(\frac{f_1,\cdots,f_n}{g}\right)$ with $f_n=\varpi^N$ be a rational subset. Note that 
    \begin{align*}
        U\left(\frac{f_1,\cdots,f_N}{g}\right)=\bigcap\limits_{i=1}^{n-1}U\left(\frac{f_i,\varpi^N}{g}\right)
    \end{align*}
    and hence we can assume $U=U\left(\frac{f,\varpi^N}{g}\right)$. By approximation lemma \ref{Approximation Lemma} for $f=f_i,c=N$ and any $\epsilon\in (0,1)$ and respectively for $f=g,c=N$, and $\epsilon=1$, there exists $a,b\in R^\flat$ such that
    \begin{align*}
        |g(x)-b^\#(x)|&<\max(|g(x)|,|\varpi|^N) ,\\
        \max(|f(x)|,|\varpi|^N)&=\max(|a^\#(x)|,|\varpi|^N).
    \end{align*}
    Then it is easy to see that $U\left(\frac{f,\varpi^N}{g}\right)=U\left(\frac{a^\#,\varpi^N}{b^\#}\right)$. This proves that every rational subset of $X$ is the inverse image of rational subsets of $X^\flat$. Because $X$ is $T_0$, the map $\phi$ is injective. We have $x\in X^\flat$ factors as composite $R^\flat\to \hat{k(x)}\to \Gamma\cup \{0\}$ with $\hat{k(x)}$ a perfectoid field. This untilt to a perfectoid field $\widehat{k(x^\flat)}$ gives a valuation on $R^\flat$.

    (ii) The above proof also gives the proof of (ii) for general $K_\Delta$.
\end{proof}

\begin{corollary}
    Let $(R,R^+)$ be a perfectoid affinoid $K_\Delta$-algebra, with tilt $(R^\flat,R^{\flat +})$, and let $X=\spa(R,R^+), X^\flat=\spa(R^\flat,R^{\flat +})$. Then for all rational $U\subset X$, the pair $(\mathcal{O}_X(U),\mathcal{O}^+_X(U))$ is a perfectoid affinoid $K_\Delta$-algebra with tilt $(\mathcal{O}_{X^\flat}(U^\flat),\mathcal{O}^+_{X^\flat}(U^\flat))$.
\end{corollary}

\begin{proof}
    The proof is exactly analogous to Corollary 6.8 of \cite{Sch}.
\end{proof}

 We now state our main theorem as in \cite{Sch}.

\begin{theorem}
    Let $(R,R^+)$ be a perfectoid $K_\Delta-$algebra, and let $X=\spa(R,R^+)$. Then the associated presheaves $\mathcal{O}_X,\mathcal{O}^+_X$ are sheaves. \footnote{It is true from \cite{Buzz} that any stably uniform $(R,R^+)$ is sheafy, which also suffices for the result in this case.} 
\end{theorem}

\begin{proof}
     It suffices to show that for an arbitrary finite open cover $\{U_i\}_{i\in I}$ of $X=\mathrm{Spa}(R,R^+)$, the \v{C}ech complex.
\begin{align}\label{Cech eq}
0\to\mathcal{O}_X(X)\to \prod\limits_{i\in I}\mathcal{O}_X(U_i)\to \prod\limits_{i,j\in I}\mathcal{O}_X(U_i\cap U_j)\to\cdots
\end{align}
is exact. Taking completed tensor product with $K_s$ for $s\in\mathfrak{S}$, we get the following \v{C}ech complex
\begin{align*}
0\to\mathcal{O}_X(X)\otimes_{K_\Delta}K_s\to \prod\limits_{i\in I}(\mathcal{O}_X(U_i)\otimes_{K_\Delta} K_s)\to \prod\limits_{i,j\in I}(\mathcal{O}_X(U_i\cap U_j)\otimes_{K_\Delta} K_s)\to\cdots
\end{align*}
is an exact sequence of perfectoid $K_s$ algebras which is exact by the sheaf property of usual perfectoid spaces over $K_s$. This is true for every $s\in\mathfrak{S}$. Thus, we have exactness of the following complex
\begin{align*}
  0\to\prod\limits_{s\in\mathfrak{S}}\mathcal{O}_X(X)\otimes_{K_\Delta}K_s\to \prod\limits_{s\in\mathfrak{S}}\prod\limits_{i\in I}(\mathcal{O}_X(U_i)\otimes_{K_\Delta} K_s)\to \prod\limits_{s\in\mathfrak{S}}\prod\limits_{i,j\in I}(\mathcal{O}_X(U_i\cap U_j)\otimes_{K_\Delta} K_s)\to\cdots
\end{align*}
is an exact sequence of $K_\Delta$-modules. This exact sequence can be rewritten as 
\begin{align*}\label{s Cech eq}
  0\to\prod\limits_{s\in\mathfrak{S}}\mathcal{O}_X(X)\otimes_{K_\Delta}K_s\to \prod\limits_{i\in I}\prod\limits_{s\in\mathfrak{S}}(\mathcal{O}_X(U_i)\otimes_{K_\Delta} K_s)\to \prod\limits_{i,j\in I}\prod\limits_{s\in\mathfrak{S}}(\mathcal{O}_X(U_i\cap U_j)\otimes_{K_\Delta} K_s)\to\cdots
\end{align*}
Since each of the terms of the sequence \ref{Cech eq} is a submodule of the corresponding term of the equation \ref{s Cech eq}, we get the exactness of \ref{Cech eq}.
\end{proof}

\begin{definition}
    Let $K_\Delta$ be of characteristics $p$. A perfectoid affinoid $K_\Delta$-algebra $(R,R^+)$ is called \emph{$p$-finite} if there exists a reduced affinoid $K_\Delta$-algebra $(S,S^+)$ of topologically finite type such that $R^+=\varinjlim\limits_{\Phi}S^+$ and $R=R^+[\varpi^{-1}]$.
\end{definition}

\begin{theorem}\label{IntegralTateAcyclicity}
    Let $(R,R^+)$ be a perfectoid $K_\Delta$-algebra and $(R^\flat,R^{\flat+})$ be its tilt. Assume $(R^\flat,R^{\flat+})$ is the completion of filtered direct limit of $p-$finite perfectoid affinoid $K^\flat_\Delta-$algebras and suppose $X=\spa(R,R^+)$. Then for any finite covering $X=\bigcup\limits_{i\in I}U_i$ by finitely many rational subsets, the sequence 
    \begin{align*}
        0\to \mathcal{O}_X(X)^{\circ a}\to \prod\limits_{i\in I}\mathcal{O}_X(U_i)^{\circ a}\to \prod\limits_{i,j\in I}\mathcal{O}_X(U_i\cap U_j)^{\circ a}\to\cdots
    \end{align*}
    is exact. In particular, $\cH^i(X,\mathcal{O}^{\circ a}_X)=0$.
\end{theorem}

\begin{proof}
    The same argument in Proposition 6.11, Lemma 6.13, and Proposition 6.14 of \cite{Sch} shows that it is enough to prove that if $(R,R^+)$ is topologically of finite type; the corresponding \v{C}ech cohomology group is annihilated by some power of $\varpi$. We know that the \v{C}ech complex
\[
\begin{tikzcd}
0\arrow[r] & \mathcal{O}_X(X)\arrow[r,"d_0"] & \prod\limits_{i\in I}\mathcal{O}_X(U_i)\arrow[r,"d_1"] & \prod\limits_{i,j\in I}\mathcal{O}_X(U_i\cap U_j)\arrow[r,"d_2"] &\cdots
\end{tikzcd}
\]
is exact. We also have that $\ker(d_i)$ is a closed subspace of a $K_\Delta$-Banach space, hence it is itself a $K_\Delta$-Banach space, and $d_{i-1}$ is a surjection onto $\ker(d_i)$. Thus, by the Open Mapping Theorem \ref{Open Mapping Theorem}, we get that $d_{i-1}$ is an open map to $\ker(d_i)$. Thus, the subspace and quotient topologies on $\ker(d_i)=\mathrm{im}(d_{i-1})$ coincide. Consider the sequence of power-bounded elements obtained from the above \v{C}ech complex:
\[
\begin{tikzcd}
0\arrow[r] &\mathcal{O}_X(X)^\circ \arrow[r,"d^\circ_0"] &\prod\limits_{i\in I}\mathcal{O}_X(U_i)^\circ \arrow[r,"d^\circ_1"] & \prod\limits_{i,j\in I}\mathcal{O}_X(U_i\cap U_j)^\circ \arrow[r,"d^\circ_2"] & \cdots
\end{tikzcd}
\]
If $R$ is topologically of finite type, then the quotient topology on $\mathrm{im}(d_{i-1})$ has $\{\varpi^{m_1}_1\mathrm{im}(d_{i-1})+\cdots+ \varpi^{m_t}_t\mathrm{im}(d_{i-1}) \mid (m_1,\cdots,m_t)\in\mathbb{N}^t\}$ as a basis of neighborhoods of $0$. This means that the cohomology is annihilated by some power of $\varpi$. 
\end{proof}

We can finally give the definition of Perfectoid spaces over $K_\Delta$ as in \cite{Sch}, Definition 6.15, Definition 6.16.

\begin{definition}
   \begin{enumerate}
       \item[(i).] A \emph{perfectoid space} over $K_\Delta$ is an adic space over $K_\Delta$ that is locally isomorphic to affinoid perfectoid spaces. Morphisms of perfectoid spaces are morphisms of adic spaces. 
       \item[(ii).] For a perfectoid space $X$ over $K_\Delta$, the \emph{tilt} of $X$ denoted by $X^\flat$ is a perfectoid space over $K^\flat_\Delta$ such that there is a functorial isomorphism $\Hom(\spa(R^\flat,R^{\flat +}),X^\flat)=\Hom(\spa(R,R^+),X)$ for all affinoid $K_\Delta-$algebras $(R,R^+)$ with tilt $(R^\flat,R^{\flat +})$.
   \end{enumerate}
\end{definition}

\begin{proposition}
   Any perfectoid space $X$ over $K_\Delta$ admits a unique tilt $X^\flat$, unique up to unique isomorphism. This induces an equivalence of categories between the category of perfectoid spaces over $K_\Delta$ and the category of perfectoid spaces over $K^\flat_\Delta$. The underlying topological spaces $X$ and $X^\flat$ are naturally identified. A perfectoid space $X$ is affinoid perfectoid if and only if its tilt $X^\flat$ is affinoid perfectoid. For any affinoid perfectoid subspace $U\subset X$, the pair $(\mathcal{O}_X(U),\mathcal{O}^+_X(U))$ is a perfectoid affinoid $K_\Delta$-algebra with tilt $(\mathcal{O}_{X^\flat}(U^\flat),\mathcal{O}^+_{X^\flat}(U^\flat))$. Fiber products exist in the category of perfectoid spaces.
\end{proposition}

\begin{proof}
    The proof moves along the same line as Proposition 6.17, Proposition 6.18 of \cite{Sch}.
\end{proof}


\section{Almost Purity Theorem for Perfectoid Spaces over $K_\Delta$}

Throughout the section, we fix a perfectoid $\Delta-$field $K_\Delta$. We will now prove the almost purity in the category of perfectoid spaces over $K_\Delta$. To do this, let us recall the definitions given in the section \S7 of \cite{Sch} in our context. 

\begin{definition}[\cite{Sch}; Definition 7.1 and Definition 7.2]
    Let $K_\Delta$ be a perfectoid $\Delta-$field.
    \begin{enumerate}
        \item[(i)] A morphism $(R,R^+)\to (S,S^+)$ of affinoid $K_\Delta-$algebras is called \emph{finite \'etale} if $S$ is a finite \'etale $R$-algebra with the induced topology, and $S^+$ is the integral closure of $R^+$ in $S$.
        \item[(ii)] A morphism $f:X\to Y$ of adic spaces over $K_\Delta$ is called \emph{finite \'etale} if there is a cover of $Y$ by open affinoids $V\subset Y$ such that the preimage $U=f^{-1}(V)$ is affinoid, and the associated morphism of affinoid $k-$algebras
        \begin{align*}
            (\mathcal{O}_Y(V),\mathcal{O}^+_Y(V))\to (\mathcal{O}_X(U),\mathcal{O}^+_X(U))
        \end{align*}
        is finite \'etale.
        \item[(iii)] A morphism $f:X\to Y$ of adic spaces over $K_\Delta$ is called \'etale if for any pont $x\in X$ there are open neighborhoods $U$ and $V$ of $x$ and $f(x)$ and a commutative diagram
        \[
        \begin{tikzcd}
            U\arrow[r,"j"]\arrow[dr,swap,"f|_U"] & W\arrow[d,"p"]\\
                                            & V
        \end{tikzcd}
        \]
        where $j$ is an open immersion and $p$ is finite \'etale.
    \end{enumerate}
    \end{definition}

    \begin{proposition}\label{Elkik}
    \begin{enumerate}
        \item[(i)] Let $A$ be a flat $K^\circ_\Delta-$algebra such that $A$ is Henselian along $I_\varpi=(\varpi)$. Then the categories of finite \'etale $A[\varpi^{-1}]$ and finite \'etale $\hat{A}[\varpi^{-1}]$-algebras are equivalent, where $\hat{A}$ is the $\varpi$-adic completion of $A$.
        \item[(ii)] Let $A_i$ be a filtered direct system of complete flat $K^\circ_\Delta$-algebras, and let $A$ be the completion of the direct limit, which is again a complete flat $K^\circ_\Delta$-algebra. Then we have an equivalences of categories
        \begin{align*}
            A[\varpi^{-1}]_{\fet}\cong 2-\varinjlim A_i[\varpi^{-1}]_{\fet}.
        \end{align*}
        In particular, if $R$ is a filtered direct system of perfectoid $K_\Delta-$algebras and $R$ is the completion of their direct limit, then $R_{\fet}\cong S_{\fet}$.
    \end{enumerate}
    \end{proposition}

    \begin{proof}
        (i) is immediate from Proposition 5.4.53 of \cite{ART}, and for (ii) same proof as in Lemma 7.5 of \cite{Sch} works.
    \end{proof}

    \begin{theorem}[Almost Purity for Perfectoid Spaces over $K_\Delta$]
        Let $(R,R^+)$ be a perfectoid affinoid $K_\Delta-$algebra, and let $X=\spa(R,R^+)$ with tilt $X^\flat$. There is an equivalence of categories between
        \begin{align*}
            X_{\fet}\cong X^\flat_{\fet}.
        \end{align*}
        In particular, for any finite \'etale cover $S/R$, $S$ is perfectoid and $S^{\circ a}$ is finite \'etale over $R^{\circ a}$. Moreover, $S^{\circ a}$ is a uniformly almost finitely generated $R^{\circ a}-$module.
    \end{theorem}

    \begin{proof}
        Since the completed stalk $\widehat{k(x)}$ is a perfectoid field over $K_s$ for some $s\in\mathfrak{S}$. Its tilt is the completed stalk $\widehat{k(x^\flat)}$ is also a perfectoid field over $K^\flat_s$. But from the theorem of Fontaine-Wintenberger \cite{FonWin}, $\widehat{k(x)}_{\fet}$ and $\widehat{k(x^\flat)}_{\fet}$ are equivalent. Using Proposition \ref{Elkik} (ii), we have
        \begin{align*}
            2-\varinjlim\limits_{x\in U}(\mathcal{O}_X(U))_{\fet}\cong 2-\varinjlim\limits_{x\in U}(\mathcal{O}_{X^\flat}(U^\flat))_{\fet}
        \end{align*}
        By Lemma 8.2.17 (i) of \cite{KL}, the global section functor 
        \begin{align*}
            (\text{ \'etale adic spaces over $X$ } )&\to (\text{ \'etale algebra over $\Gamma(X,\mathcal{O}_X)$ })\\
              Y &\mapsto \Gamma(Y,\mathcal{O}_Y)
        \end{align*}
    induces an equivalence of categories as in \cite{Sch}, Theorem 7.9. Combining the results obtained in \S 6, we have a chain of equivalence
    \begin{align*}
        R^\flat_{\fet}\cong R_{\fet}\leftarrow R^{\circ a}_{\fet}\cong (R^{\circ a}/\varpi)_{\fet}=(R^{\flat\circ a}/\varpi^\flat)_{\fet}\cong R^{\flat\circ a}_{\fet}\cong R^\flat_{\fet}.
    \end{align*}
    Therefore, if $S$ is finite \'etale over $R$, then $S^\flat$ is finite \'etale over $R^\flat$ and hence following the above chain of equivalence we obtain $S^{\circ a}$ is finite \'etale over $R^{\circ a}$. This chain of equivalence also proves that $S^{\circ a}$ is an almost finitely generated $R^{\circ a}-$module.
    \end{proof}

    This completes the diagram. 

    \[
    \begin{tikzcd}
    R_{\fet} &\arrow[l,swap,"\cong"] A_{\fet}\arrow[d]\arrow[r,"\cong"] & \overline{A}_{\fet}\arrow[d] & \arrow[l,swap,"\cong"] A^\flat_{\fet}\arrow[d] \arrow[r,"\cong"] & R^\flat_{\fet}\arrow[d]\\
   K_\Delta-\Perf  &\arrow[l,swap,"\cong"]K^{\circ a}_\Delta-\Perf\arrow[r,"\cong"] & (K^{\circ a}_\Delta/\varpi)-\Perf &\arrow[l,swap,"\cong"] K^{\flat\circ a}_\Delta-\Perf &\arrow[l,swap,"\cong"] K^\flat_\Delta-\Perf 
   \end{tikzcd} 
   \]

  \begin{example}[Multivariate perfectoid open unit disk]
     Let $R=K^\circ_\Delta[[T^{1/p^\infty}]]$ be the $(\varpi,T)-$adic completion of $K^\circ_\Delta[T^{1/p^\infty}]$. Then, following \cite{Weins}, we can define the \emph{perfectoid open unit disk} $\mathbb{D}$ over $K_\Delta$ as $\spa (R,R^\circ)\setminus \{\varpi=0\}$. Note that $R$ is not a $K_\Delta-$algebra, but we claim that $\mathbb{D}$ is a perfectoid space over $K_\Delta$. Indeed, if we set $R_n=R\left<\frac{T^n}{\varpi}\right>$, then $R^\circ_n=R\left<\left(\frac{T^n}{\varpi}\right)^{1/p^\infty}\right>$, which we consider as $R^+_n$. Then $\mathbb{D}=\spa(R,R^\circ)\setminus\{\varpi=0\}=\bigcup\limits_{n=1}^\infty \spa(R_n,R^+_n)$. This is because $\spa(R_n,R^+_n)$ is the rational subset $U\left(\frac{T^n,\varpi}{\varpi}\right)$. Conversely, suppose $x\in \spa(R,R^\circ)$ with $|\varpi(x)|\neq 0$. Then, since $T$ is topologically nilpotent, there exists $n\in\mathbb{N}$ such that $|T^n(x)|\leq |\varpi(x)|$. Hence, $x\in U\left(\frac{T^n,\varpi}{\varpi}\right)$ for some $n$.
  \end{example}


\bibliographystyle{amsalpha}
\bibliography{name}

\end{document}